\documentclass[12pt]{article}
\usepackage{graphicx}
\usepackage{amscd, amssymb}

\newenvironment{eq}{\begin{equation}}{\end{equation}}
\newenvironment{proof}{{\bf Proof}:}{\vskip 5mm }

\newtheorem{proposition}{Proposition}[section]
\newtheorem{lemma}[proposition]{Lemma}
\newtheorem{definition}[proposition]{Definition}
\newtheorem{theorem}[proposition]{Theorem}
\newtheorem{cor}[proposition]{Corollary}

\newtheorem{example}[proposition]{Example}
\newtheorem{examples}[proposition]{Examples}
\newtheorem{remark}[proposition]{Remark}

\newcommand{\llabel}[1]{\label{#1}}
\newcommand{\comment}[1]{}
\newcommand{\sr}{\rightarrow}

\newcommand{\bdl}{\bar{\Delta}}
\newcommand{\tdl}{\bar{\Delta}}
\newcommand{\zz}{{\bf Z\rm}}

\newcommand{\oo}{\otimes}

\newcommand{\af}{{\bf A}^1}

\newcommand{\dsr}{\stackrel{\sr}{\scriptstyle\sr}}

\newcommand{\BB}{_{\bullet}}
\newcommand{\RR}{{\bf R}}
\newcommand{\LL}{{\bf L}}

{\vskip
3mm}

\begin{document}

\begin{center}
{\Large\bf Simplicial radditive functors}\\ 
\vspace{3mm} 

{\large\bf Vladimir Voevodsky}\footnote{School of Mathematics, Institute for Advanced Study,
Princeton NJ, USA. e-mail: vladimir@ias.edu}\\
\vspace{3mm}

{September 2009}\\
\end{center}
\tableofcontents
\parskip= 0.2in
\section{Introduction}
Let us start with the following observation:

{\em Let $F:Sets\sr Sets$ be any functor which commutes with directed colimits. Then its extension to the category of simplicial sets takes weak equivalences to weak equivalences} 

The goal of this paper is construct a framework which can be used to proof results of this kind for a wide class of closed model categories and functors between those categories. Originally I was interested in the proof that the symmetric power functors respect $\af$-equivalences between simplicial schemes but it soon became clear that similar problems arise for other categories (such as finite correspondences or pointed schemes) and other functors (such as the forgetting functor from correspondences to schemes) and that a new toolbox is required to tackle these problems. 

Let $C$ be a category with finite coproducts $\amalg$ and an initial object $\emptyset$. A contravariant functor $F:C\sr Sets$ is called radditive if $F(\emptyset)=pt$ and for any $X, Y$ in $C$ the natural map $F(X\amalg Y)\sr F(X)\times F(Y)$ is a bijection. In model theory, categories which are equivalent to the categories of radditive functors are known as finitary varieties.  Examples of such categories include categories of presheaves, categories of pointed presheaves and categories of additive functors on additive categories.
 
The category of simplicial objects in the category $Rad(C)$ of radditive functors on $C$ carries a natural finitely generated simplicial c.m.s. called the projective c.m.s. Let $H(C)$ be its homotopy category. For any set of morphisms $E$ in $\Delta^{op}Rad(C)$ one defines in the usual way the class of (left) $E$-local equivalences $cl_l(E)$. The localization $H(C)[cl_l(E)^{-1}]$ always exists  and we denote it by $H(C,E)$.  If the projective c.m.s. is left proper then by Smith's localization theorem $cl_l(E)$ is the class of weak equivalences of $E$-local local c.m.s. which is the (left) Bousfield localization of the projective c.m.s. but many important results about $H(C,E)$ hold without the left properness assumption.

Let $C^{\#}$ be the full subcategory in $Rad(C)$ which consists of directed colimits of representable functors. Any continuous functor (i.e. a functor which commutes with directed colimits) $F:C^{\#}\sr (C')^{\#}$ defines a functor $F^{rad}:Rad(C)\sr Rad(C)$ which is not a left adjoint and does not commute with colimits unless $F$ commutes with coproducts. Given two sets of morphisms $E$ and $E'$ in $\Delta^{op}Rad(C)$  and $\Delta^{op}Rad(C')$ respectively, we want to find a natural condition on $E$, $E'$ and $F$ and a sufficiently wide class of objects in $\Delta^{op}Rad(C)$ such that the simplicial extension of $F^{rad}$ takes $E$-local equivalences between objects of this class to $E$-local equivalences.

Since $F^{rad}$ does not commute with colimits the class of cofibrant objects which is usually considered in the constructions of homotopy derived functors does not play any particular role in our approach. Instead one considers the class $\Delta^{op}C^{\#}$. The standard cofibrant replacement functor $Cof$ of the projective c.m.s. takes values in $\Delta^{op}C^{\#}$ which implies that any cofibrant object is a retract of an object in $\Delta^{op}C^{\#}$ but most objects of $\Delta^{op}C^{\#}$ are neither fibrant nor cofibrant. Since any object of $\Delta^{op}Rad(C)$ is projectively equivalent to an object of $\Delta^{op}C^{\#}$ we may consider $H(C)$ as a localization of the later category. 

Corollary \ref{2009forin} of our first functoriality theorem asserts that projective equivalences between objects of $\Delta^{op}C^{\#}$ are preserved by the simplicial extension of $F$ for any continuous functor $F:C^{\#}\sr (C')^{\#}$. This is a direct generalization of the result about simplicial sets mentioned at the beginning to categories of radditive functors. From it we conclude that any continuous $F:C^{\#}\sr (C')^{\#}$ defines in a natural way a functor 
$${\bf L}F^{rad}:H(C)\sr H(C').$$

Given sets of morphisms $E$ and $E'$ in $\Delta^{op}Rad(C)$ and $\Delta^{op}Rad(C')$ respectively we may now ask for a natural condition on $E$, $E'$ and $F$ which would guarantee that ${\bf L}F^{rad}$ takes $E$-local equivalences to $E'$-local equivalences. It is done in Theorem \ref{2007el1} which asserts that for a continuous functor $F:C^{\#}\sr (C')^{\#}$ and a set of morphisms $E$ in $\Delta^{op}C^{\#}$ such that for any $f\in E$ and $X\in C$ the morphism $F(f\coprod Id_X)$ is in $cl_l(E')$ one has 
\begin{eq}
\llabel{2009eq1}
F(cl_l(E)\cup \Delta^{op}C^{\#})\subset cl_l(E')
\end{eq}
and in particular
\begin{eq}
\llabel{2009eq2}
{\bf L}F^{rad}(cl_l(E))\subset cl_l(E').
\end{eq}
When $F$ is a functor $C\sr C'$ which commutes with finite coproducts the functor $F^{rad}$ has a right adjoint $F_{rad}$ and Theorem \ref{2007eadj} asserts that under the obvious necessary conditions $F^{rad}$ respects $E$-local equivalences between objects from $\Delta^{op}C^{\#}$ and $F_{rad}$ respects $E'$-local equivalences between all objects of $\Delta^{op}Rad(C')$. 

The main technical tool which we use in the proofs is the notion of a $\tdl$-closed class of morphisms in the category of $\Delta^{op}C$ of simplicial objects over a category $C$ (see Definition \ref{2009td}) and the related notion of $\tdl$-closure $cl_{\tdl}(E)$. As follows immediately from their definition, $\tdl$-closures commute with simplicial extensions of all continuous functors. 

Our functoriality results are obtained from this property of $\tdl$-closure and Theorems \ref{2009main} and \ref{2009main2} which express $E$-local equivalences in $\Delta^{op}C^{\#}$ and $\Delta^{op}Rad(C)$ respectively in terms of $\tdl$-closures.

I am very grateful to Charles Weibel who made a great number of useful suggestions both for the original version of the paper and for the new one.

\section{Elementary properties of $\Delta$-closed classes}
\subsection{$\Delta$-closed classes}
\label{section1}
Let $C$ be a category and $\Delta^{op}C$ the category of simplicial
objects over $C$. Following \cite{Swan2}, define a unit homotopy
from a morphism $f:A\sr B$ to a morphism $g:A\sr B$ in $\Delta^{op}C$
as a collection of morphisms $h_i^n:A_n\sr B_n$ where $n\ge 0$ and
$i=-1,\dots,n$ satisfying the following conditions:
\begin{enumerate}
\item $h^n_{-1}=f_n$, $h^n_n=g_n$ where $f_n$ and $g_n$ are the
components of $f$ and $g$
\item $\partial_ih_j=h_{j-1}\partial_i$ if $i\le j$,
$\partial_ih_j=h_{j}\partial_i$ if $i>j$
\item $s_ih_j=h_{j+1}s_i$ if $i\le j$, $s_ih_j=h_js_i$ if $i>j$.
\end{enumerate}
If $C$ has coproducts (resp. finite coproducts), $K$ is a set (resp. a finite set) and $X$ an object of
$C$ we let $X\oo K=\amalg_K X$ denote the coproduct of $K$ copies of $X$. Similarly for a
simplicial set (resp. finite simplicial set) $K$ and an object $X$ of $\Delta^{op}C$ we let   $X\oo K$ denote the simplicial object with terms $X_n\oo K_n$. 
\begin{example}\rm
If $C$ is the category of sets then $X\oo K=X\times K$. If $C$ is the category of pointed sets then $X\oo K=X\wedge (K_+)$.
\end{example}
One
verifies easily (see \cite[Prop. 2.1]{Swan2}) that if $C$ has finite coproducts then a unit homotopy from $f$
to $g$ in the sense of the definition given above is the same as a
morphism $h:A\oo\Delta^1\sr B$ such that $h\circ
(Id\oo\partial_0)=f$ and $h\circ
(Id\oo\partial_1)=g$. 

Two morphisms are called homotopic if they can be connected by a chain
of unit homotopies (going in either direction). A morphism
$f:A\sr B$ in $\Delta^{op}C$ is called a homotopy equivalence if there
exists a morphism $g:B\sr A$ such that the compositions $gf$ and $fg$
are homotopic to the corresponding identity morphisms.
\begin{definition}
\llabel{fd} Let $\cal C$ be a category. A
class $E$ of morphisms in $\Delta^{op}{\cal C}$ is called
$\Delta$-closed if the following conditions hold.
\begin{enumerate}
\item All  homotopy equivalences are in $E$. 
\item If $f$ and $g$ are morphisms such that the composition $gf$ is
defined and two out of three morphisms $f, g, gf$ are in $E$ then the
third is in $E$.
\item If $f:B\sr B'$ is a morphism of bisimplicial objects over $C$
such that the rows or columns of $f$ are in $E$ then the diagonal morphism
$\Delta(f)$ is in $E$
\end{enumerate}
\end{definition}
We denote the smallest $\Delta$-closed class containing a
class $E$ by $cl_{\Delta}(E)$.
\begin{remark}
\rm\llabel{2009r1}
The definitions of a $\Delta$-closed and $\bdl$-closed classes given here are not equivalent to the definitions of classes with the same names in \cite{Delnotessub}. However the reader should have no problem connecting these definitions to each other.
\end{remark} 
A functor $F:C\sr D$ defines in the usual way a functor $\Delta^{op}C\sr \Delta^{op}D$ which we denote by the same symbol $F$ and call the simplicial extension of $F$.  The simplicial extensions 
take homotopic morphisms to homotopic morphisms and if we define
$F$ on bisimplicial objects by setting $(F(X))_{ij}=F(X_{ij})$ we have
$F\circ \Delta=\Delta\circ F$. This implies the following result.
\begin{lemma}
\llabel{fandd} Let $F:C\sr C'$ be a functor. Then for any class of morphisms $E$
in $\Delta^{op}C$ one has
$$F(cl_{\Delta}(E))\subset cl_{\Delta}(F(E))$$
\end{lemma}
\begin{proposition}
\llabel{fcop} Let $C$ be a category with finite coproducts and $E$ a class of morphisms in $\Delta^{op}C$. Let $E\amalg Id_{C}$ be the class of morphisms of the form$f\amalg Id_{X}$ for $f\in E$ and $X\in C$. Then $cl_{\Delta}(E\amalg Id_C)$ is closed under finite coproducts.
\end{proposition}
\begin{proof}
For a pair of morphisms $f_1:X_1\sr
Y_1$, $f_2:X_2\sr Y_2$ we have 
$$f\amalg g=(Id_{Y_1}\amalg f_2)\circ (f_1\amalg
Id_{X_2})$$
and therefore it is sufficient to check that for a morphism $f$ in
$cl_{\Delta}(E\amalg Id)$ and an object $X$ in $\Delta^{op}C$ one has
$f\amalg Id_X\in cl_{\Delta}(E\amalg Id)$. Since $f\amalg Id_X$ is the diagonal of a morphism whose rows are of the form $f\amalg Id_{X_i}$ for $X_i\in C$ it is sufficient to show that for $f\in cl_{\Delta}(E\amalg Id)$ and $X\in C$ one has $f\amalg Id_X\in cl_{\Delta}(E\amalg Id)$. This follows from Lemma \ref{fandd} applied to the functor $(-)\amalg Id_X$.
\end{proof}
Proposition \ref{fcop} shows that $cl_{\Delta}(E\amalg Id_C)$ is the smallest class which contains $E$ and is $(\Delta,\amalg_{<\infty})$-closed i.e. $\Delta$-closed and closed under finite coproducts. We will denote it by $cl_{\Delta,\amalg_{<\infty}}(E)$.
\begin{cor}
\llabel{fcope}
Let $C$ be a category with finite coproducts. Then 
$$cl_{\Delta,\amalg_{<\infty}}(\emptyset)=cl_{\Delta}(\emptyset)$$
\end{cor}

\begin{definition}
\llabel{strict} A morphism $e:A\sr X$ in a category $C$ is called a
coprojection if it is isomorphic to the canonical morphism $A\sr A\amalg Y$ for some $Y$. A morphism $f:A\sr X$ in $\Delta^{op}C$ is
called a term-wise coprojection if for any $i\ge 0$ the morphism
$f_i:A_i\sr X_i$ is a coprojection.
\end{definition}
For any morphism
$f:B\sr A$ and any object $X$, the square
\begin{equation}
\label{elexe0}
\begin{CD}
B@>e_B>>B\amalg X\\
@VVV @VVV\\
A@>e_A>>A\amalg X
\end{CD}
\end{equation}
is a push-out square. This shows that in a category with finite
coproducts there exist push-outs for all pairs of morphisms
$(e,f)$ such that $e$ is a coprojection. Therefore the same is true
for pairs of morphisms $(e,f)$ in $\Delta^{op}C$ such that $e$ is a
term-wise coprojection. 
\begin{definition}
\llabel{elex}
A square in $\Delta^{op}C$ is called an
elementary push-out square if it is isomorphic to the push-out square
for a pair of morphisms $(e,f)$ where $e$ is a term-wise coprojection.
\end{definition}
Suppose that $C$ has finite coproducts. For any commutative square $Q$ of the form 
\begin{equation}
\label{elexe1}
\begin{CD}
B@>e_B>>Y\\
@VuVV @VVvV\\
A@>e_A>>X
\end{CD}
\end{equation}
denote by $K_Q$ the object defined by the elementary push-out square
\begin{equation}
\llabel{kq}
\begin{CD}
B\amalg B @>>> B\oo\Delta^1\\
@VVV @VVV\\
A\amalg Y @>>> K_Q
\end{CD}
\end{equation}
and by $p_Q:K_Q\sr X$ the obvious morphism. For a morphism $f:X\sr X'$
the object $K_{Q}$ defined by the square
\begin{equation}
\llabel{elexe2}
\begin{CD}
X @>f>> X'\\
@VVV @VVV\\
X @>f>> X'
\end{CD}
\end{equation}
is called the cylinder of $f$ and denoted by $cyl(f)$.
\begin{lemma}
\llabel{cyl}
The morphisms $X'\sr cyl(f)$ and $cyl(f)\sr X'$ are mutually inverse
homotopy equivalences. 
\end{lemma}
\begin{proof}
The object $cyl(f)$ is defined by the elementary
push-out square
\begin{equation}
\begin{CD}
X@>Id\oo\partial_0>>X\oo\Delta^1\\
@VVV @VVV\\
X'@>>>cyl(f)
\end{CD}
\end{equation}
The composition $X'\sr cyl(f)\sr X'$ is the identity. The homotopy
from the identity of $cyl(f)$ to the composition $cyl(f)\sr X'\sr
cyl(f)$ is given by the morphism $cyl(f)\oo\Delta^1\sr cyl(f)$
which equals to the projection $X'\oo\Delta^1\sr X'$ on
$X'\oo\Delta^1$ and to the morphism $Id\oo h$ on
$$(X'\oo\Delta^1)\oo\Delta^1=X'\oo(\Delta^1\times\Delta^1)$$
where $h:\Delta^1\times\Delta^1\sr \Delta^1$ is the usual homotopy
from the identity to the composition
$\Delta^1\sr\Delta^0\stackrel{\partial_0}{\sr}\Delta^1$.
\end{proof}
\begin{lemma}
\llabel{lm1} For an elementary push-out square $Q$ of the form
(\ref{elexe1}) in $\Delta^{op}C$, the morphism $p_Q:K_Q\sr X$ belongs to
$cl_{\Delta}(\emptyset)$.
\end{lemma}
\begin{proof}
The object $K_Q$ is the diagonal of the bisimplicial object whose rows
are $K_{Q_i}$ where $Q_i$ is the square formed by the i-th terms of
$A$, $B$ and $Y$ and $p_Q$ is the diagonal of the morphism whose terms
are $p_{Q_i}:K_{Q_i}\sr X_i$. Therefore, it is sufficient to prove the
statement of the lemma for a square in $C$ of the form
(\ref{elexe0}). Since $K_{Q_1\amalg Q_2}=K_{Q_1}\amalg K_{Q_2}$ and
a square of the form (\ref{elexe0}) is a coproduct of a square of the
form (\ref{elexe2}) and a transpose of such a square our result
follows from Lemma \ref{cyl} and the fact that the coproduct of two simplicial homotopy equivalences is a simplicial homotopy equivalence.
\end{proof}
\begin{lemma}
\llabel{lm0}
Let $f=(f_A,f_B,f_Y,f_X):Q\sr Q'$ be a morphism of commutative squares of the form
(\ref{elexe1}) then one has
$$(K(f):K_{Q}\sr K_{Q'})\in cl_{\Delta,\amalg_{<\infty}}(\{f_A, f_B,
f_Y\})$$
\end{lemma}
\begin{proof}
The object $K_Q$ is isomorphic to
the diagonal of a bisimplicial object whose rows are of the form $A\amalg
Y\amalg (\amalg_{n} B)$ and this isomorphism is natural with respect to  morphisms of squares.
\end{proof}
Combining Lemma \ref{lm1} and Lemma \ref{lm0} we get the following result.
\begin{proposition}
\llabel{2009short}
Let $f=(f_A,f_B,f_Y,f_X):Q\sr Q'$ be a morphism of elementary push-out squares of the form
(\ref{elexe1}). Then one has
$$f_X\in cl_{\Delta,\amalg_{<\infty}}(\{f_A, f_B,
f_Y\})$$
\end{proposition}

\begin{lemma}
\llabel{exactsquares} 
Let $C$ be a category with finite
coproducts. Then for any elementary push-out square in 
$\Delta^{op}C$ of the form (\ref{elexe1}) one has
$$e_A\in cl_{\Delta, \amalg_{<\infty}}(\{e_B\})$$
$$v\in cl_{\Delta, \amalg_{<\infty}}(\{u\})$$
\end{lemma}
\begin{proof}
To prove the first inclusion one applies Proposition \ref{2009short} to the morphism of squares of the
form 
$$
\left(
\begin{CD}
B @>>> B\\
@VVV @VVV\\
A @>>> A
\end{CD}
\right)
\longrightarrow
\left(
\begin{CD}
B @>>> Y\\
@VVV @VVV\\
A @>>> X
\end{CD}
\right)
$$
To prove the second inclusion one applies Proposition \ref{2009short}  to the
morphism of squares
$$
\left(
\begin{CD}
B @>>> Y\\
@VVV @VVV\\
B @>>> Y
\end{CD}
\right)
\longrightarrow
\left(
\begin{CD}
B @>>> Y\\
@VVV @VVV\\
A @>>> X
\end{CD}
\right)$$
\end{proof}
\begin{lemma}
\llabel{twosq}
Consider a commutative diagram of the form
\begin{equation}
\begin{CD}
B @>>> Y @>>> Z\\
@VVV  @VVV @VVV\\
A @>>> X @>>>T
\end{CD}
\end{equation}
Denote the left square by $Q_1$, the right square by $Q_2$ and the
big square by $Q_3$. Consider the canonical morphisms
$$\begin{array}{ccc}
p_1:K_{Q_1}\sr X&p_2:K_{Q_2}\sr T&p_3:K_{Q_3}\sr T
\end{array}$$
and let $E$ be a $(\Delta,\coprod_{<\infty})$-closed class. If 
two out of three morphisms $p_1,p_2,p_3$ are in $E$ then the third is
in $E$.
\end{lemma}
\begin{proof}
Let $Q_4$ be the square
\begin{equation}
\begin{CD}
Y @>>> Z\\
@VVV @VVV\\
K_{Q_1} @>>> K_{Q_3}
\end{CD}
\end{equation}
One can easily see that it is elementary push-out. The identity morphisms
$Y\sr Y$, $Z\sr Z$ and the morphism $p_1:K_{Q_1}\sr X$ define a
morphism of squares $f:Q_4\sr Q_2$ and we get a commutative diagram
\begin{equation}
\begin{CD}
K_{Q_4} @>K(f)>> K_{Q_2}\\
@Vp_{4}VV @VVp_{2}V\\
K_{Q_3} @>p_{3}>> T
\end{CD}
\end{equation}
By Lemma \ref{lm1} we have $p_{4}\in
cl_{\Delta,\coprod_{<\infty}}(\emptyset)$ and by Lemma \ref{lm0} we have
$K(f)\in cl_{\Delta,\coprod_{<\infty}}(p_{1})$. This implies the
statement of the lemma.
\end{proof}

\subsection{$\tdl$-closed classes}
\begin{definition}
\llabel{filtcolnew}
A class of morphisms $E$ in a category $C$ is said to be closed under filtered colimits if for any pair of filtered systems $(X_i)_{i\in I}$, $(Y_i)_{i\in I}$ such that $X=colim_i X_i$ and $Y=colim_i Y_i$ exist and any  morphism of systems $(f_i):(X_i)\sr (Y_i)$ such that $f_i\in E$ one has $f=colim_i f_i\in E$.
\end{definition}
We will say that a functor between any two categories $C\sr C'$ is continuous (or finitary) if it commutes with filtered colimits which exist in the first category. As shown in \cite[p. 15]{Adamek} a functor is continuous if and only if it commutes with directed colimits, and a functor whose domain is a category with directed colimits is continuous if and only if it preserves colimits of chains. 

Let us recall the following definition (see \cite[p. 452]{Beke}).
\begin{definition}
\llabel{2009tc}
Let $C$ be a category and $E$ be a class of morphisms in $C$. A morphism $f:X\sr Y$ is called a transfinite composition of morphisms from $E$ if there is an ordinal $\alpha$ and a continuous functor $F:\alpha\sr C$ such that $colim F$ exists, $f$ is isomorphic to the morphism $F(0)\sr colim F$ and for any $i\in \alpha$ the morphism $F(i\sr i+1)$ is in $E$.
\end{definition}
\begin{lemma}\llabel{obv2}
Let $E$ be a class closed under finite compositions and filtered colimits. Then one has:
\begin{enumerate}
\item for any filtered system $i\mapsto X(i)$ such that for all $i\sr j$ the morphism $X(i)\sr X(j)$ is in $E$ and any $i\in I$, the map $X(i)\sr colim_i X(i)$ is in $E$,
\item  $E$ it is closed under transfinite compositions. 
\end {enumerate}
\end{lemma}
\begin{proof}
To prove the first assertion observe first that by replacing $I$ with the category $i/I$ of morphisms $i\sr j$ and our original system with the system $(i\sr j)\mapsto X_j$ we do not change the colimit. The first assertion of the lemma follows now by the application of Definition \ref{filtcolnew} to the obvious morphism from the constant system $(i\sr j)\mapsto X_i$ to the system $(i\sr j)\mapsto X_j$.

To prove the second assertion let $I$ be an ordinal considered as a partially ordered set and $i\mapsto X(i)$ be a functor such that $X$ takes all successor arrows to elements of $E$ and such that for $j=\lim_{i<j} i$ one has $X(j)=colim_{i<j} X(i)$. Consider the subset $J\subset I$ which consists of $J$ such that $X(0)\sr X(j)$ is in $E$. If $j=colim_{i<j} i$ and for all $i<j$ we have $i\in J$ then by the first assertion $j\in J$. Since $E$ is closed under finite compositions we conclude that if $j\in J$ then $j+1\in J$. Therefore $J=I$ by the "transfinite induction" axiom.
\end{proof}
\begin{definition}
\llabel{2009td}
Let $C$ be a category. A class in $\Delta^{op}C$ is called $\tdl$-closed if it is $\Delta$-closed and closed under filtered colimits.
\end{definition}
\begin{remark}\rm
As follows from \cite[Th. 1.5, p.14]{Adamek} filtered colimits are equivalent to directed colimits in a way which makes these two notions interchangeable in all of our results and definitions.
\end{remark}
We let $cl_{\tdl}(E)$ denote the smallest $\tdl$-closed class which contains $E$. There is the following obvious analog of Lemma \ref{fandd}.
\begin{lemma}
\llabel{fandt}
Let $F:C\sr C'$ be a continuous functor. Then for any class of morphisms $E$ in $\Delta^{op}C$ one has
$$F(cl_{\tdl}(E))\subset cl_{\tdl}(F(E))$$
\end{lemma}
\begin{proposition}
\llabel{2009fcop}
Let $C$ be a category with coproducts and $A$ a class of objects in $C$ which is closed under finite coproducts and such that any object of $C$ is a filtered colimit of objects of $A$. Then for any class of morphisms $E$ in $\Delta^{op}C$ the class $cl_{\tdl}(E\amalg Id_A)$ is closed under coproducts.
\end{proposition}
\begin{proof}
Let $f_i:X_i\sr Y_i$, $i\in I$ be a set of morphisms in $\Delta^{op}C$. Choose a well-ordering on $I$ and for $j\in I+1$ set
$$Z_j=(\amalg_{i<j} Y_i)\amalg (\amalg_{i\ge j} X_i)$$
then $j\mapsto Z_j$ is a transfinite composition diagram such that $Z_j\sr Z_{j+1}$ is isomorphic to $f_j\amalg Id_{(\amalg_{i<j} Y_i)\amalg (\amalg_{i<j} X_i)}$ and $Z_0\sr colim Z_j$ is isomorphic to $\amalg f_i$. Therefore it remains to show that for $f\in cl_{\tdl}(E\amalg Id_A)$ and $X\in \Delta^{op}C$ one has $f\amalg Id_X\in cl_{\tdl}(E\amalg Id_A)$. Since $f\amalg Id_X$ is the diagonal of the morphism of bisimplicial objects whose rows are of the form $f\amalg X_i$ we may assume that $X\in C$. By our assumption on $A$ the morphism $f\amalg Id_X$ is a filtered colimit of morphisms of the form $f\amalg Id_Y$ for $Y\in A$. Therefore it remains to show that for $f\in cl_{\tdl}(E\amalg Id_A)$ and $Y\in A$ one has $f\amalg Id_Y\in cl_{\tdl}(E\amalg Id_A)$. This follows from Lemma \ref{fandt} applied to the functor $(-)\amalg Y$.
\end{proof}
Proposition \ref{2009fcop} shows that $cl_{\tdl}(E\amalg Id_C)$ is the smallest class which is $\tdl$-closed and closed under coproducts. We denote it by $cl_{\tdl,\amalg}(E)$.
\begin{cor}
\llabel{2009fcope}
Let $C$ be a category with coproducts. Then 
$$cl_{\tdl,\amalg}(\emptyset)=cl_{\tdl}(\emptyset)$$
\end{cor}
\begin{theorem}
\llabel{diag} The class of weak equivalences in $\Delta^{op}Sets$ coincides with
$cl_{\tdl}(\emptyset)$.
\end{theorem}
\begin{proof}
Let us show first that any weak equivalence $f$ belongs to $cl_{\tdl}(\emptyset)$. 
Let $Ex=colim_n Ex_n$ be the Kan completion functor. For a weak equivalence $f$ the map $Ex(f)$ is a homotopy equivalence and by the 2-out-of-3-property it remains to show that the maps $X\sr Ex(X)$ are in $cl_{\tdl}(\emptyset)$.  Since any $cl_{\tdl}(\emptyset)$ is
closed under countable compositions it is sufficient
to show that the maps
$$Ex_n(X)\sr Ex_{n+1}(X)$$
belong to $cl_{\tdl}(\emptyset)$. By definition of
$Ex_n$ (see e.g. \cite{GabZis}) we have a square of the form
$$
\begin{CD}
\amalg_{\Lambda^{n,k}\sr Ex_n(X)}
\Lambda^{n,k} @>>> 
Ex_n(X)\\ 
@VVV @VVV\\
\amalg_{\Lambda^{n,k}\sr Ex_n(X)}
\Delta^n @>>> 
Ex_{n+1}(X)
\end{CD}
$$
which is clearly elementary push-out. The left hand side vertical arrow is a homotopy equivalence and by Lemma \ref{exactsquares} the right hand side vertical arrow is in $cl_{\Delta, \amalg_{\le \infty}}(\emptyset)=cl_{\Delta}(\emptyset)$. 

To show that $cl_{\tdl}(\emptyset)$ is contained in the class of weak equivalences it is sufficient to verify that the class of weak equivalences is $\tdl$-closed.  The conditions of Definition \ref{fd} are well known (for the proof of the third condition see e.g. \cite[Lemma 5.3.1 p.129]{Hovey}, \cite[Prop. 1.7, p.199]{GJ}). The fact that the class of weak equivalences is closed under
filtered colimits follows easily from the
fact that $Ex$ commutes with filtered colimits and the definition of weak equivalences of Kan simplicial sets in terms of homotopy groups. 
\end{proof}
\begin{cor}
\llabel{unexpect} 
Let $C$ be a category with small coproducts, $f:K\sr K'$ a weak equivalence of simplicial sets and $X$ an object of $C$. Then one has
$$(Id_{X}\oo
f:X\oo K\sr
X\oo K')\in cl_{\tdl}(\emptyset)$$ 
\end{cor}
\begin{proof}
It follows from Theorem \ref{diag} and Lemma \ref{fandt} applied to the functor $K\mapsto X\oo K$.
\end{proof} 
Let $I$ be a small category and $C$ be a category with coproducts. Let further $F:I\sr C$ be a functor.  Define a simplicial object $hocolim(F)$ in $C$ setting 
$$hocolim(F)_n=\coprod_{i_0\sr\dots\sr i_n} F(i_0)$$
where the coproduct is taken over all sequences of morphisms in $I$ of length $n+1$. One defines the face and degeneracy morphisms in the obvious way (it is the same construction as in \cite{BKan} or in \cite[p.199]{GJ} only now the target category is $C$ instead of $Sets$).  Note that 
$$hocolim(F)_0=\coprod_{i\in I} F(i)$$
and in particular for any $i$ in $I$ we have a morphism $F(i)\sr hocolim(F)$. There is also an obvious morphism $hocolim(F)\sr colim(F)$ if $colim(F)$ exists. 

If $F$ is a functor $I\sr \Delta^{op}C$ then $hocolim(F)$ is a bisimplicial object. By abuse of notation we will often write $hocolim(F)$ instead of $\Delta(hocolim(F))$. The definition of $hocolim$ immediately implies the following lemma. 
\begin{lemma}
\llabel{2009li}
Let $C$ be a category with coproducts and $I$ a small category. Let further $F, G: I \sr \Delta^{op}C$ be two functors and $f:F\sr G$ a natural transformation. Then one has
$$hocolim(f)\in cl_{\tdl,\amalg}(\{f(i)\}_{i\in I})$$
\end{lemma}
\begin{proposition}
\llabel{2009hcol}
\llabel{import1}
Let $C$ be a category with coproducts and $F:I\sr \Delta^{op}C$ a filtered diagram such that $colim(F)$ exists. Then one has
$$(hocolim(F)\sr colim(F))\in cl_{\tdl}(\emptyset).$$
\end{proposition}
\begin{proof}
Since our morphism is the diagonal of a morphism of bisimplicial objects whose rows are of the form $hocolim(F_n)\sr colim(F_n)$ for $F_n:I\sr C$ we may assume that $F$ takes values in $C$. 

Let $I_0$ be the set of objects of $I$ considered as a category where all morphisms are identities. We have an obvious functor $\phi:I_0\sr I$ which defines a functor $\phi_*:Funct(I,C)\sr Funct(I_0,C)$. If $C$ has coproducts then $\phi_*$ has a left adjoint $\phi^*$ which sends a family $(X_i)_{i\in I}$ to the functor 
$$\phi^*((X_i)_{i\in I_0}):j\mapsto \amalg_{i\sr j} X_i$$
where the coproduct is over all morphisms $i\sr j$ in $I$. The adjoint pair $(\phi^*,\phi_*)$ defines in the usual way a cotriple $\Phi=\phi^*\phi_*$ on $Funct(I,C)$. For any $F\in Funct(I,C)$ we get a simplicial functor $\Phi_*(F)$ with terms  $\Phi_n(F)=(\phi_*\phi^*)^{\circ (n+1)}(F)$. An elementary computation shows that 
$$\Phi_*(F)(j)=\phi_*\Phi_*(F)(j)=hocolim(F/j)$$
where $F/j:(i\sr j)\mapsto F(i)$. By the standard properties of simplicial objects defined by co-triples we conclude that  the obvious map $hocolim(F/j)\sr colim(F/j)=F(j)$ is a simplicial homotopy equivalence. 

Suppose now that $I$ is a filtered category. Then 
$$hocolim(F)=colim(j\mapsto hocolim(F/j))$$
and the morphism $hocolim(F)\sr colim(F)$ is a filtered colimit of homotopy equivalences and therefore belongs to $cl_{\tdl}(\emptyset)$. 
\end{proof}
\begin{proposition}
\llabel{bdlchar}
Let $C$ be a category with coproducts and filtered colimits. Then a class $E$ in $\Delta^{op}C$ is $(\tdl,\amalg)$-closed if and only if it is $\Delta$-closed, closed under coproducts and contains $cl_{\tdl}(\emptyset)$.
\end{proposition}
\begin{proof}
The "only if" part is obvious. To prove the "if" part it is sufficient by \cite[Th. 1.5, p.14]{Adamek} to check that a class $E$ satisfying the conditions of the proposition is closed under directed colimits which follows immediately from Proposition \ref{import1} and Lemma \ref{2009li}.
\end{proof}

\section{Homotopy theory of simplicial radditive\\ functors}
\subsection{Radditive functors}
\llabel{rf}
Let $C$ be a category with finite coproducts and an initial object
$0$. Denote by $Rad(C)$ the full subcategory of the category of
contravariant functors from $C$ to sets which consists of functors $F$
such that $F(0)=pt$ and for any finite family of objects $X_{i}$, $i\in I$ the
map
$$F(\amalg_{i\in I} X_i)\sr \prod_{i\in I} F(X_i)$$
is bijective. The objects of $Rad(C)$ will be called radditive
functors. 

Categories of radditive functors can be also thought of as categories of covariant functors on categories with finite products which respect the products. Such categories are known in model theory as finitary varieties. See \cite[p.132]{Adamek}.

\begin{examples}\rm\llabel{2009exs1}
\begin{enumerate}
\item \llabel{st1}\rm
Recall that a presheaf on a small category is a contravariant functor
from this category to the category of sets. Let $C$ be a small
category and $C^{\amalg_{<\infty}}$ the full subcategory of the
category of presheaves on $C$ which consists of finite coproducts of
representable presheaves. Then $Rad(C^{\amalg_{<\infty}})$ is
equivalent to the category of presheaves on $C$.
\item\llabel{st2}\rm
For an object $X$ of a category $C$ let $X_+$ be the pointed presheaf
on $C$ obtained from the presheaf represented by $X$ by the addition
of a disjoint base point. Let $C^{\amalg_{<\infty}}_+$ be the full
subcategory of the category of pointed presheaves on $C$ which
consists of coproducts of objects of the form $X_+$. Then
$Rad(C^{\amalg_{<\infty}}_+)$ is the category of pointed presheaves on
$C$.
\item \llabel{st0}\rm
If $C$ is an additive category then $Rad(C)$ is equivalent to the
abelian category of additive contravariant functors from $C$ to abelian
groups. 
\item \llabel{alg}
Let $A$ be a commutative ring and $C$ the category of finitely
generated free algebras over $A$. Then $Rad(C)$ is
equivalent to the category of all  algebras over $A$.  A similar result holds for categories of finitely generated free groups, etc. See Proposition \ref{2009mon}.
\item \llabel{union}
Let $C_1$, $C_2$ be two categories with finite coproducts. Then $C_1\times C_2$ has finite coproducts given by $(X_1,X_2)\amalg (Y_1,Y_2)=(X_1\amalg Y_1, X_2\amalg Y_2)$ and the category $Rad(C_1\times C_2)$ is canonically equivalent to the category $Rad(C_1)\times Rad(C_2)$.
\end{enumerate}
\end{examples}
Example \ref{2009exs1}(\ref{st2}) has a generalization which we will state as a lemma.
\begin{lemma}
\llabel{2007pointed}
Let $C$ be a category with finite coproducts and a final object $pt$ and $C_+$ the full subcategory of pointed objects in $C$ which consists of objects of the form $X_+=X\amalg pt$. Then $Rad(C_+)$ is equivalent to the category of pointed objects in $Rad(C)$.
\end{lemma}
\begin{proof}
Note that $pt=\emptyset_+$ is the initial object in $C_+$. Therefore for any $F\in Rad(C_+)$ one has $F(pt)=pt$. Let $Rad(C)_{\BB}$ be the category of pointed objects in $Rad(C)$. Define a functor
$$\phi:Rad(C_+)\sr Rad(C)_{\bullet}$$
by the rule $\phi(F):X\sr F(X_+)$ with the distinguished point in $F(X_+)$ being the image of $F(pt)\sr F(X_+)$. 

Define a functor 
$$\psi:Rad(C)_{\bullet}\sr Rad(C_+)$$
by the rule $\phi(G):X_+\sr G(X)$. For a morphism $f:X_+\sr Y_+$ we define $\psi(G)(f)$ as the composition
$$G(Y)\sr G(Y)\times G(pt)\cong G(Y_+)\sr G(X_+)\cong G(X)\times G(pt)\sr G(X).$$
where the first morphism is the product of the identity with the distinguished point in $G(pt)$. One verifies easily that $\phi$ and $\psi$ are mutually inverse equivalences.
\end{proof}
\begin{remark}\rm
Lemma \ref{2007pointed} may be considered as a particular case of Proposition \ref{2009mon} since pointed objects are exactly algebras over the monad $X\mapsto X_+$. 
\end{remark}
Any representable functor is radditive by definition of
coproducts. Therefore we have a full embedding  of $C$ to $Rad(C)$ which sends
an object $X$ to the corresponding representable functor. 
\begin{proposition}
\llabel{2009elprop}
\begin{enumerate}
\item \llabel{rad1} The category $Rad(C)$ has all small limits.
The limit of a diagram $F:I\sr Rad(C)$ is the same as its limit in the
category of presheaves of sets i.e.
$$(lim F(i))(U)=lim (F(i)(U)).$$
\item \llabel{rad6} The functor $C\sr Rad(C)$ commutes with finite coproducts.
\item \llabel{2009rad1}
Let $f:F\sr G$ be a morphism of radditive functors and $Im(f)$ be its image in the category of presheaves. Then $Im(f)$ is a radditive functor.
\item \llabel{rad5} Let $F:I\sr Rad(C)$ be a filtered system of radditive
functors. Then the colimit $colim F(i)$ of $F$ in the category of 
functors is radditive and gives a colimit of $F$ in the category of
radditive functors.
\item \llabel{2007rad1} Recall that a coequalizer diagram $X\dsr Y$ is called {\em reflexive} (or split) if both morphisms have a common right inverse (section) $Y\sr X$.
Let $X\dsr Y$ be a reflexive coequalizer diagram of radditive functors. Then the coequalizer of this diagram in the category of presheaves is a radditive functor and a coequalizer of $X\dsr Y$ in  the category of
radditive functors.
\end{enumerate}
\end{proposition}
\begin{proof} The first three statements are obvious from definitions. 
The fourth and the fifth statements follow from the fact that filtered colimits and reflexive coequalizers of sets commute with finite products.
\end{proof}
\begin{lemma}
\llabel{rad21}
Let $X_{i}$, $i\in I$ be a collection of objects of $C$. Denote by the same
symbols the radditive functors represented by $X_{i}$. Define 
$\amalg_i X_{i}$ as the functor given by
$$U\mapsto colim_{A\subset I} Hom(U, \amalg_{i\in A} X_{i})$$
where $A$ run through finite subsets of $I$. Then $\amalg_{i}
X_{i}$ is radditive and it is a coproduct of $X_{i}$ in the
category of radditive functors.
\end{lemma}
\begin{proof}
The fact that $\amalg_i X_i$ is radditive follows from Proposition \ref{2009elprop}(\ref{rad5}) which together with \ref{2009elprop}(\ref{rad6}) also shows that 
$$\amalg_i X_i=colim_{A\subset I} \amalg_{i\in A} X_{i}$$
which obviously implies the statement of the lemma.
\end{proof}
For a set $S$ and an object $X$ of $C$ we let $X\oo S$ denote the coproduct of $S$ copies of $X$ in $Rad(C)$.
\begin{proposition}
\llabel{rad2}
The inclusion functor 
$$Rad(C)\sr Funct(C^{op}, Sets)$$
has a left adjoint.
\end{proposition}
\begin{proof}
Let $F$ be a functor which is not necessarily radditive. Consider the
coequalizer diagram in $Rad(C)$ of the form
\begin{eq}
\llabel{chuck1}
\amalg_{(p:U\sr V)\in C} U\oo F(V) \rightrightarrows
\amalg_{W\in C} W\oo F(W)
\end{eq}
where the first arrow maps the summand $U$ corresponding to $p:U\sr V$ and 
$f\in F(V)$ to the summand corresponding to $U$ and $F(p)(f))$ by the 
identity and the second one maps it to the summand corresponding to
$V$ and $f$ by $p$. Let $r(F)$ be the coequalizer of these two maps in
the category of functors. Since (\ref{chuck1}) is reflexive via $W\mapsto Id_W$  this functor is radditive by Proposition \ref{2009elprop}(\ref{2007rad1}) and one
verifies easily that for any radditive $G$ one has 
$$Hom(F, G)= Hom(r(F), G)$$
\end{proof}
\begin{example}\rm\llabel{chuck2}
The functor $r$ is not, in general,
left exact i.e. it does not commute with finite limits. In particular,
radditive functors can not be thought of as sheaves with respect to
some topology on $C$.
Let $C$ be the category of finitely generated free abelian groups so
that $Rad(C)$ is equivalent to the category of all abelian groups. Consider the functor $F$ defined by the push-out square
$$
\begin{CD}
\zz\amalg\zz @>>> \zz\times\zz\\
@VVV @VVpV\\
0 @>>> F
\end{CD}
$$
where $\zz$ is the functor represented by $\zz$ and $\amalg$ and $\times$ are in the category of all functors from $C$ to $Sets$. Let $i:\zz\sr F$ be the composition of the diagonal $\zz\sr\zz\times\zz$ with $p$. One verifies easily that it is a monomorphism. On the other hand $r(F)=0$ and therefore $r(i)$ is not a monomorphism.
\end{example}
As a corollary of Proposition \ref{2009elprop}(\ref{rad5},\ref{2007rad1}) and of the proof of Proposition \ref{rad2} we get the following characterization of radditive functors.
\begin{lemma}
\llabel{2007s1}
A contravariant functor from $C$ to $Sets$ is radditive iff it is the coequalizer of a reflexive pair whose terms are filtered colimits of representable functors.
\end{lemma}
\begin{remark}\rm
Lemma \ref{2007s1} allows one to define the notion of a radditive functor on a category $C$ without the assumption that $C$ has finite coproducts. However in this case it is unclear why $Rad(C)$ has limits or colimits.
\end{remark} 
\begin{lemma}
\llabel{rad3}
The category $Rad(C)$ (and therefore the category $\Delta^{op}Rac(C)$) is cocomplete i.e. has all  small colimits.
\end{lemma}
\begin{proof}
For a diagram $F:I\sr Rad(C)$ of radditive functors one gets $colim F$
applying the radditivization functor of Proposition \ref{rad2} to the colimit of $F$ in
the category of functors.
\end{proof}
For two simplicial radditive functors (and in particular for simplicial sets) $X$, $Y$ we let $S(X,Y)$ denote the simplicial set with terms $S(X,Y)_n=Hom(X\oo\Delta^n,Y)$.  For a simplicial set $K$ and $X\in \Delta^{op}Rad(C)$  let $Hom_{\oo}(K,X)$ be the simplicial radditive functor which takes
$U\in C$ to $S(K, X(U))$. Then $Hom_{\oo}(K,-)$ is right adjoint to $(-)\oo K$ and $\Delta^{op}Rad(C)$ is a simplicial category with respect to these structures (see \cite[Def. 2.1, p.81]{GJ}). 

Recall (see e.g. \cite{Adamek}) that an object $X$ in a category  is called compact or finitely presentable if $Hom(X,-)$ commutes with filtered (directed) colimits 
\begin{lemma}
\llabel{2009comcom}
Representable functors are compact in $\Delta^{op}Rad(C)$.  If $X$ be a compact object of $\Delta^{op}Rad(C)$ and $K$ is a finite simplicial set then the object $X\oo K$ is compact.
\end{lemma}
\begin{proof}
The fact that representable functors are compact follows immediately from Proposition \ref{2009elprop}(\ref{rad5}). 
Since $Hom_{\oo}(K,-)$ is right adjoint to $(-)\oo K$, to prove the second assertion it is sufficient,to verify that for a finite simplicial set $K$ the functor $Hom_{\oo}(K,-)$ commutes with filtered colimits.  This follows immediately from the definition of $Hom_{\oo}(K,-)$ and Proposition \ref{2009elprop}(\ref{rad5}).
\end{proof}
A category is called locally finitely presentable if it is cocomplete and there is a set $\cal A$ of finitely presentable objects such that every object is a directed colimit of objects from $\cal A$. 
\begin{proposition}
\llabel{2009pr1}
The category $\Delta^{op}Rad(C)$ is locally finitely presentable.
\end{proposition}
\begin{proof}
By \cite[Th. 1.11, p.17]{Adamek} a category is locally finitely presentable if and only if it is cocomplete and has a strong generator (see  \cite[p.2]{Adamek}) which consists of compact objects. That the category $\Delta^{op}Rad(C)$ is cocomplete follows from Lemma \ref{rad3}. The set of objects of the form $X\oo\Delta^n$ for $X\in C$ and $n\ge 0$ forms a strong generator of $\Delta^{op}Rad(C)$ and by Lemma \ref{2009comcom}, the objects $X\oo\Delta^n$ are compact.
\end{proof}

\subsection{Projective closed model structure}
\llabel{pcms}
\begin{definition}
\llabel{2009peq}
Let $C$ be a category with finite coproducts. A morphism $f:X\sr Y$ of simplicial radditive  functors is called a projective  equivalence if for any $U$ in $C$, the map of
simplicial sets $X(U)\sr Y(U)$ defined by $f$ is a weak
equivalence.
\end{definition}
We denote the class of projective equivalences by $W_{proj}$. 
\begin{example}\rm\llabel{2009exs3}
\begin{enumerate}
\item
\llabel{2007st1}
The equivalence 
$$Rad(C^{\amalg_{<\infty}})\sr Funct(C^{op},Sets)$$
of Example \ref{2009exs1}(\ref{st1}) identifies projective equivalences of simplicial radditive functors on $C^{\amalg_{<\infty}}$ with section-wise equivalences of simplicial presheaves on $C$.
\item
\llabel{2007st2}
The equivalence
$$Rad(C_+)\sr Rad(C)_{\bullet}$$
of Lemma \ref{2007pointed} identifies projective equivalences of simplicial radditive functors on $C_+$ with projective equivalences of pointed simplicial radditive functors on $C$.
\item \llabel{2007st0} The equivalence of Example \ref{2009exs1}(\ref{st0}) identifies projective equivalences with quasi-isomorphisms of the corresponding normalized complexes. 
\item \llabel{2007alg} The equivalences of Example \ref{2009exs1}(\ref{alg}) identify projective equivalences with the usual notion of weak equivalences for simplicial algebras, groups etc. 
\end{enumerate}
\end{example}
\begin{example}\rm
\llabel{2007grnot}
The radditivization functor $r$  need not take projective equivalences of simplicial presheaves to projective equivalences of radditive functors. Let $i$ be a morphism of Example \ref{chuck2}. Since it is a monomorphism of presheaves the natural morphism $p:cone(i)\sr \pi_0(cone(i))$ is a weak equivalence of presheaves. The radditification $r(p)$ of this morphism is not a projective equivalence since $r(i)$ is not a monomorphism and therefore $r(cone(i))=cone(r(i))$ has a non-trivial $\pi_1$. 
\end{example} 
\begin{proposition}
\llabel{2009tdl0}
The class of projective equivalences is $\tdl$-closed and in particular contains $cl_{\tdl}(\emptyset)$. 
\end{proposition}
\begin{proof}
It follows from Theorem \ref{diag} and Lemma \ref{fandt} applied to the functors of sections over $U\in C$.
\end{proof}
\begin{cor}
\llabel{2007ss}
For an object $X$ of $\Delta^{op}Rad(C)$ and a weak equivalence of simplicial sets $K\sr L$ the morphism $X\oo K\sr X\oo L$ is a projective equivalence.
\end{cor}
\begin{proof}
It follows from Proposition \ref{2009tdl0} and Corollary \ref{unexpect}.
\end{proof}
Let $C_0$ be the set of objects of $C$ considered as a category where all morphisms are identities and let $\phi:C_0\sr C$ be the obvious functor. Then $\phi$ defines a pair of adjoint functors between $Funct(C_0^{op},Sets)$ and $Funct(C^{op},Sets)$. Composing the functors of this pair with the inclusion $Rad(C)\sr Funct(C^{op},Sets)$ and its left adjoint we get a pair of functors 
$$\phi_r:Rad(C)\sr Funct(C_0^{op},Sets)$$
$$\phi^l:Funct(C_0^{op},Sets)\sr Rad(C)$$
where $\phi_r$ is the right adjoint and $\phi^l$ the left adjoint. Consider the cotriple $L=\phi^l\phi_r$ on $Rad(C)$ defined by this adjunction. An object of $Funct(C_0^{op},Sets)$ is a family of sets $(F_U)_{U\in C}$ parametrized by objects of $C$ and one has
$$\phi^l((F_U)_{U\in C})=\amalg_{U\in C} U\oo F_U$$
where on the right hand side $U$ is considered as an object of $Rad(C)$. Therefore for $F\in Rad(C)$ we have 
\begin{eq}
\llabel{2009lres1}
L(F)=\amalg_{U\in C} U\oo F(U)
\end{eq}
In particular, $L$ takes values in $\Delta^{op}\bar{C}$ where $\bar{C}$ is the full subcategory in $Rad(C)$ which consists of coproducts of representable functors. Let $L_*$ be the functor $Rad(C)\sr \Delta^{op}\bar{C}$ defined in the standard way by the cotriple $L$. 
\begin{proposition}
\llabel{2009lres3}
Let $C$ be a category with finite coproducts. Then one has:
\begin{enumerate}
\item for any $U\in C$ the morphism $L_*(U)\sr U$ is a simplicial homotopy equivalence in $\Delta^{op}\bar{C}$,
\item for any  $F\in Rad(C)$ and $U\in C$ the morphism $L_*(F)(U)\sr F(U)$ is a homotopy equivalence of simplicial sets.
\end{enumerate}
\end{proposition}
\begin{proof}
For $U\in C$ the radditive functor $U$ is $\phi^l(U)$ and for $F\in Rad(C)$ and $U\in C$ the simplicial morphism $L_*(F)(U)\sr X(U)$ is obtained by evaluation on $U$ of the morphism $\phi_r(L_*(F)\sr F)$. Therefore both assertions of the proposition follow from the standard properties of the simplicial objects associated with cotriples.
\end{proof}
Applying $L_*$ to a simplicial radditive functor $X$ we get a bisimplicial radditive functor which we also denote by $L_*(X)$. A simple explicit  computation shows the columns of $L_*(X)$ are of the form:
\begin{eq}
\llabel{2009lres2}
(L_*(X))_n=\amalg_{U_0\sr\dots\sr U_n} U_0\oo X(U_n)
\end{eq}
where the coproduct is taken over all sequences of arrows in $C$ of length $n$. 

Let $C^{\#}$ be the full subcategory of $Rad(C)$ which consists of filtered colimits of representable functors. Note that $\bar{C}\subset C^{\#}$.
\begin{proposition}
\llabel{2009sum1}
One has:
\begin{enumerate}
\item\llabel{lresfilt} $\Delta L_*$ commutes with filtered colimits and reflexive coequalizers,
\item\llabel{lresd} for any $X\in \Delta^{op}C^{\#}$ the morphism $\Delta L_*(X)\sr X$ is in $cl_{\tdl}(\emptyset)$,
\item\llabel{all} for any $X\in \Delta^{op}Rad(C)$ the morphism $\Delta L_*(X)\sr X$ is in $W_{proj}$,
\item\llabel{naoborot} $\Delta L_*$ takes projective  equivalences to
elements of $cl_{\tdl}(\emptyset)$.
\end{enumerate}
\end{proposition}
\begin{proof}
The first assertion follows from Proposition \ref{2009elprop}(\ref{rad5}, \ref{2007rad1}) and (\ref{2009lres2}). The second follows from the first and Proposition \ref{2009lres3}(1). The third follows from Proposition \ref{2009lres3}(2) and Proposition \ref{2009tdl0}. The fourth one follows from (\ref{2009lres2}) and Corollary \ref{unexpect}.
\end{proof}
\begin{remark}\rm
The functor $L_*$ does not commute with finite coproducts. Example where $C$
is the category of finitely generated free abelian groups and
$Rad(C)$ is the category of all abelian groups shows that in general it is not
possible to find a functor from $Rad(C)$ to $\Delta^{op}C^{\#}$
which commutes with finite coproducts and satisfies \ref{2009sum1}(\ref{all}).
\end{remark}
\begin{theorem}
\llabel{coin} The class of morphisms in $\Delta^{op}C^{\#}$ which are
projective  equivalences as morphisms of simplicial radditive functors
coincides with $cl_{\tdl}(\emptyset)$. In particular it is
$(\tdl,\amalg)$-closed.
\end{theorem} 
\begin{proof}
The second statement follows from the first one and Corollary \ref{2009fcope}.
By Proposition \ref{2009tdl0} it is sufficient to show that any projective equivalence $f:X\sr Y$ between objects of $\Delta^{op}C^{\#}$ lies in $cl_{\tdl}(\emptyset)$. 
Consider the square
$$
\begin{CD}
{\Delta L_*}(X) @>>> {\Delta L_*}(Y)\\
@VVV @VVV\\
X @>f>> Y
\end{CD}
$$
The vertical arrows are in $cl_{\tdl}(\emptyset)$ by Proposition \ref{2009sum1}(\ref{lresd}). The upper horizontal arrow is in $cl_{\tdl}(\emptyset)$ by Proposition \ref{2009sum1}(\ref{naoborot}). Therefore the lower horizontal arrow is in $cl_{\tdl}(\emptyset)$.
\end{proof}
Let $\Delta_{Mon}$ be the subcategory of monomorphisms in the standard
simplicial category $\Delta$. A contravariant functor from
$\Delta_{Mon}$ to a category is a ``simplicial object with no
degeneracies'' (also called a semi-simplicial object, see \cite[Section 8.1.9, 8.1.10]{W}). Let $\pi_*$ be the obvious forgetful functor from
$\Delta^{op}C$ to $\Delta^{op}_{Mon}C$. A general argument shows that if
$C$ has colimits then $\pi_*$ has a left adjoint $\pi^*$. In fact,
since any morphism in $\Delta$ has a canonical decomposition into an
epimorphism followed by a monomorphism, one needs only finite
coproducts to define $\pi^*$. For a functor $Z=(Z_i)$ from
$\Delta_{Mon}$ to $C$ the simplicial object $\pi^*(Z)$ has terms of
the form 
$$\pi^*(Z)_i=\amalg_{[i]\sr [j]} Z_j$$
where $[i]\sr [j]$ runs through epimorphisms from $[i]$ to $[j]$ in
$\Delta$ (see \cite[Ex.8.1.5]{W}). An object $X$ in
$\Delta^{op}Rad(C)$ is called degeneracy-free if it belongs to the image of
this functor.

If $Z=(Z_n)$ is an object of $\Delta^{op}_{Mon}$ and $X=\pi^*(Z)$ the
corresponding degeneracy-free simplicial object we say that $X$ is
based on $(Z_n)$.  The composition $Wr=\pi^*\pi_*$ is called the wrapping functor. For any simplicial object $X$ the object $Wr(X)$ is degeneracy-free and its terms are given by the formula
\begin{equation}
\llabel{eq7.28.1}
Wr(X)_i=\amalg_{[i]\sr [j]} X_j
\end{equation}
where $[i]\sr [j]$ runs through epimorphisms from $[i]$ to $[j]$ in
$\Delta$ and $X_j$ are the terms of $X$. For any
$X$, the adjunction defines a natural morphism $Wr(X)\sr X$.
\begin{theorem}
\llabel{2009innov}
For any $X\in \Delta^{op}C^{\#}$ the morphism $p_X:Wr(X)\sr X$ is a projective equivalence.
\end{theorem}
\begin{proof}
Consider $L_*(X)$ as a bi-simplicial object whose rows are of the form $L_*(X_m)$. Then its n-th column is of the form 
$$\amalg_{U_0\sr\dots\sr U_n}U_0\oo X(U_n)$$
where the coproduct is in $Rad(C)$ over all sequences of morphisms $U_0\sr\dots\sr U_n$ in $C$. Let $Wr^v$ be the vertical wrapping functor  i.e. the wrapping functor applied column by column. Consider the following square:
$$
\begin{CD}
\Delta Wr^v(L_*(X)) @>a>> \Delta L_*(X)\\
@VbVV @VcVV\\
Wr(X) @>p>> X
\end{CD}
$$
Let us show that the morphisms $a,b,c$ are projective equivalences and therefore $p$ is a projective equivalence by the 2-out-of-3 property. We have $c\in W_{proj}$ by Proposition \ref{2009sum1}(\ref{all}). The morphism $b$ is the diagonal of the morphism whose $i$-th row is of the form
$$\amalg_{[i]\sr [j]} L_*(X_j)\sr \amalg_{[i]\sr [j]}X_j.$$
where the coproduct is over all the order-preserving surjections $[i]\sr [j]$. These morphisms are projective equivalences by Proposition \ref{2009sum1}(\ref{all}) and Theorem \ref{coin}. Therefore $b$ is a projective equivalence since the class of projective equivalences is $\Delta$-closed. The morphism $a$ is the diagonal of the morphism whose columns are of the form 
\begin{eq}
\llabel{innovcolumns}
\amalg_{U_0\sr\dots\sr U_n}U_0\oo Wr(X(U_n))\sr \amalg_{U_0\sr\dots\sr U_n}U_0\oo X(U_n)
\end{eq}
where $Wr(X(U_n))$ is the wrapping functor on the simplicial {\em set} $X(U_n)$. Since for a simplicial set $S$, the morphism $Wr(S)\sr S$ is a weak equivalence we conclude by Corollary \ref{2007ss} that the morphisms (\ref{innovcolumns}) are  coproducts of projective equivalences between objects of $\Delta^{op}C^{\#}$ and therefore projective equivalences by Theorem \ref{coin}. The theorem is proved. 
\end{proof}
\begin{remark}\rm
The proof of Theorem \ref{2009innov} shows that for a category $C$ such that $W_{proj}$ is closed under coproducts the morphism $p_X:Wr(X)\sr X$ is a projective equivalence for any $X$.
\end{remark}

\begin{definition}
\llabel{2009pcms}
Let $C$ be a category with finite coproducts and $f:X\sr Y$ be a morphism of simplicial radditive  functors.
\begin{enumerate}
\item  $f$ is called a projective fibration if for any $U$ in $C$, the map of
simplicial sets $X(U)\sr Y(U)$ defined by $f$ is a Kan fibration,
\item $f$ is called a projective cofibration if it has the left lifting property for morphisms which are projective fibrations and projective equivalences. 
\end{enumerate}
\end{definition}
We denote the classes of projective fibrations and cofibrations by $Fib_{proj}$ and $Cof$ respectively. 
Let $I$ be the set of morphisms of the form $U\oo\partial\Delta^n\sr U\oo\Delta^n$ for $U\in C$ and $n\ge 0$ and $J_{proj}$ be the class of morphisms of the form $U\oo \Lambda^{n,k}\sr U\oo \Delta^n$ for $U\in C$ and $\Lambda^{n,k}$ being the usual "horn" simplicial sets. 
\begin{theorem}
\llabel{2009pcmsth}
The classes $(W_{proj}, Fib_{proj}, Cof)$ form a finitely generated closed model structure on $\Delta^{op}Rad(C)$ with the classes of generating cofibrations and generating trivial cofibrations being $I$ and $J_{proj}$ respectively. 
\end{theorem}
\begin{proof}
For the definition of a finitely generated c.m.s. see \cite[Def. 2.1.17, p.34]{Hovey}. 

The domains and codomains of elements of $I$ and $J_{proj}$ are compact and a morphism is a projective fibration (resp. a projective fibration and a projective equivalence) if and only if it has the right lifting property with respect to $J_{proj}$ (resp. $I$). We will show that the classes $I$ and $J_{proj}$ satisfy the conditions of \cite[Th. 2.1.19, p.35]{Hovey} with respect to $W_{proj}$. Standard reasoning shows that elements of $J_{proj}$ are $I$-cells and in particular $I$-cofibrations and therefore $J_{proj}$-cells are $I$-cofibrations. It is also obvious that $I$-injective morphisms are $J$-injective and that they are projective equivalences and that a projective equivalence which is $J_{proj}$-injective is $I$-injective. 

It remains to show that $J_{proj}$-cells are projective equivalences. By Proposition \ref{2009tdl0} the class $W_{proj}$ is closed under transfinite compositions. Therefore it is sufficient to show that a push-out of an element of $J_{proj}$ is in $W_{proj}$. Elements of $J_{proj}$ are term-wise coprojections and simplicial homotopy equivalences. Therefore by Lemma \ref{exactsquares} a push-out of an element of $J_{proj}$ belongs to $cl_{\tdl}(\emptyset)$ and by Proposition \ref{2009tdl0} to $W_{proj}$.
\end{proof}
\begin{example}\rm
The push-out squares of \cite[Cor. 1.14, p.358]{GJ} imply that for an object $X$ of $\Delta^{op}C^{\#}$ the object $Wr(X)$ is projectively cofibrant.
\end{example}
Recall (see \cite{Hirs}) that for a class of morphisms $A$ a morphism $f$ is called a sequential $A$-cell if it is a countable composition of push-outs of coproducts of elements of $A$. Since domains and codomains of elements of $I$ (resp. $J_{proj}$) are compact we can use the small object argument to obtain for any morphism $f:X\sr Y$ a functorial decomposition of the form $X\stackrel{f_1}{\sr}\tilde{Y}\stackrel{f_2}{\sr}Y$ where $f_1$ is a sequential $I$-cell (resp. $J_{proj}$-cell) and $f_2$ has the right lifting property for $I$ (resp. $J_{proj}$). Let us define the standard cofibrant and fibrant replacement functors $Cof$ and $Fib_{proj}$ for the projective c.m.s. using this construction.  One verifies easily that if elements of $A$ are term-wise coprojections then sequential $A$-cells are term-wise coprojections as well. In particular we get the following result. 
\begin{proposition}
\llabel{2009cofrep}
One has:
\begin{enumerate}
\item the standard cofibrant replacement functor $Cof$ takes values in $\Delta^{op}\bar{C}$,
\item any cofibration is a domain preserving retract of a coprojection whose terms are of the form $X\sr X\amalg Y$ for $Y\in \bar{C}$, 
\end{enumerate}
\end{proposition}

The following proposition gives an important sufficient condition for a class of morphism in $\Delta^{op}Rad(C)$ to be $\tdl$-closed. In this proposition we let $CofEnds$ denote the class of morphisms between the cofibrant objects.
\begin{proposition}
\llabel{2009.07.23.1}
Let $C$ be a small category with finite coproducts and $E$ a class of morphisms in $\Delta^{op}Rad(C)$ which satisfies the following conditions:
\begin{enumerate}
\item $E$ contains $W_{proj}$,
\item $E$ satisfies the 2-out-of-3 property (i.e. the second condition of Definition \ref{fd}),
\item $E\cap \Delta^{op}C^{\#}\cap CofEnds$ is closed under coproducts,
\item for $f\in E\cap \Delta^{op}C^{\#}\cap CofEnds$ and $i\ge 1$ one has $f\oo Id_{\partial\Delta^i}\in E$,
\item for a morphism of push-out squares
$$
\left(
\begin{CD}
f_1\,\, @. f_2\\
f_3\,\, @. f_4
\end{CD}
\right)
:
\left(
\begin{CD}
X_1 @>>> X_2\\
@VgVV @VVV\\
X_3 @>>> X_4
\end{CD}
\right)
{\longrightarrow}
\left(
\begin{CD}
X_1' @>>>X_2'\\
@Vg'VV @VVV\\
X_3' @>>> X_4'
\end{CD}
\right)
$$
such that all the objects are in $\Delta^{op}C^{\#}\cap CofEnds$, the morphisms $g$, $g'$ are cofibrations and term-wise coprojections and $f_1,f_2,f_3$ are in $E$, one has $f_4\in E$,
\end{enumerate}
Then $E$ is $\tdl$-closed.
\end{proposition}
\begin{proof}
The first two condition of Definition \ref{fd} follow from the first two conditions of our proposition. 
To prove the third one we start with the following lemma.
\begin{lemma}
\llabel{2009.07.23.2}
Let $E$ be a class satisfying the conditions of the proposition. Then for two sequences of morphisms $(X_i\stackrel{g_{i+1}}{\sr} X_{i+1})_{i\ge 0}$, $(X_i'\stackrel{g_{i+1}}{\sr} X_{i+1}')_{i\ge 0}$ and a morphism of sequences $f_i:X_i\sr X_i'$ such that $X_i,X_i'\in \Delta^{op}C^{\#}\cap CofEnds$ and $f_i\in E$, the morphism $colim f_i:colim\,X_i\sr colim\,X_i'$ is in $E$.
\end{lemma}
\begin{proof}
Let us recall in our context the standard construction of a telescope for the sequence $(X_i\stackrel{g_{i+1}}{\sr}X_{i+1})_{i\ge 0}$. Let $\iota_i$ be the canonical morphism $X_i\sr \amalg_{i\ge 0} X_i$. Consider the morphism 
$$(\amalg_{i\ge 0}X_i)\oo\partial\Delta^1\sr \amalg_{i\ge 0}X_i$$
which is identity on $(\amalg_{i\ge 0}X_i)\oo\{0\}$ and the "shift" morphism $\amalg_{i\ge 0}\iota_{i+1}g_{i+1}$ on   $(\amalg_{i\ge 0}X_i)\oo\{1\}$.  The telescope of the sequence $(g_i)_{i\ge 1}$ is defined by the elementary push-out square
\begin{eq}\llabel{2009.07.23.eq1}
\begin{CD}
(\amalg_{i\ge 0}X_i)\oo\partial\Delta^1 @>>> \amalg_{i\ge 0}X_i\\
@VVV @VVV\\
(\amalg_{i\ge 0}X_i)\oo\Delta^1 @>>> Tel_{\infty}((g_i))
\end{CD}
\end{eq}
There is an obvious morphism $Tel_{\infty}((g_i))\sr colim_{i} X_i$. We claim that this morphism belongs to $cl_{\tdl}(\emptyset)$. Indeed, consider the partial telescopes given by
$$
\begin{CD}
(\amalg_{n>i\ge 0}X_i)\oo\partial\Delta^1 @>>> \amalg_{n\ge i\ge 0}X_i\\
@VVV @VVV\\
(\amalg_{n>i\ge 0}X_i)\oo\Delta^1 @>>> Tel_{n}((g_i))
\end{CD}
$$
There are obvious morphisms $Tel_n\sr Tel_{n+1}$ and $colim_n Tel_n=Tel_{\infty}$. One further observes that $Tel_0=X_0$, $Tel_1=cyl(g_1)$ and more generally that for each $n$ there is a simplicial homotopy equivalence $Tel_n((g_i))\sr X_n$ and these equivalences form a morphism of sequences $(Tel_n)_{n\ge 0}\sr (X_n)_{n\ge 0}$ whose colimit is the morphism $Tel_{\infty}((g_i))\sr colim_i X_i$ which is, therefore, in $cl_{\tdl}(\emptyset)$. 

To finish the proof of the lemma it remains to show that the morphism $Tel_{\infty}((g_i))\sr Tel_{\infty}((g_i'))$ is in $E$. It follows from the conditions 3,4 and then 5 applied to the morphism of squares of the form (\ref{2009.07.23.eq1}) defined by the family $(f_i)$ (note that for $f:X\sr Y$ in $E$ the morphism $f\oo Id_{\Delta^i}$ is in $E$ by the first two properties of $E$). 
\end{proof}

Let $f:B\sr B'$ be a morphism of bisimplicial
radditive functors with columns $f_i:B_{i*}\sr B_{i*}'$ in $E$.  Applying the standard cofibrant replacement
functor to each column of $B$ and $B'$ we get a commutative square of the
form 
$$
\begin{CD}
Cof_{c}(B) @>>> Cof_{c}(B')\\
@VVV @VVV\\
B @>>> B'
\end{CD}
$$
where the columns of $Cof_c(B)$ and $Cof_c(B')$ are cofibrant and belong to $\Delta^{op}C^{\#}\cap CofEnds$.  In view of Proposition \ref{2009tdl0} the vertical arrows define projective 
equivalences on the diagonal objects. Therefore, in order to prove the third property of Definition \ref{fd} it is sufficient to show that the diagonal of the upper horizontal arrow is in $E$.

Let $Wr_rCof_c(B)$ be the bisimplicial object obtained by the application of the wrapping functor to each row of $Cof_c(B)$.  Its columns are of the form $Z_i=\amalg_{[i]\sr [j]} Cof(B_{j*})$. In particular they are cofibrant and belong to $\Delta^{op}C^{\#}$. The diagonal of the projection $Wr_rCof_c(B)\sr Cof_c(B)$ is in $W_{proj}$ by Theorem \ref{2009innov} and Proposition \ref{2009tdl0}. The same construction applies to $B'$. It remains to show that the morphism
\begin{eq}
\llabel{2007diag1}
\Delta(Wr_rCof_c(B))\sr \Delta(Wr_rCof_c(B'))
\end{eq}
is in $E$.  By \cite[Cor. 1.14, p.358]{GJ} applied to the simplicial objects formed by columns, the morphism  (\ref{2007diag1}) is the colimit of the sequence of morphisms 
\begin{eq}
\llabel{skim}
\Delta(sk_i(Wr_rCof_c(B))) \sr \Delta(sk_i(Wr_rCof_c(B)))
\end{eq}
which fit into morphisms of elementary push-out squares of the form
$$
\left(
\begin{CD}
Z_i\oo\partial \Delta^i @>>> \Delta(sk_{i-1}(Wr_{r}Cof_c(B)))\\
@VVV @VVV\\
Z_i\oo\Delta^i @>>> \Delta(sk_i(Wr_{r}Cof_c(B)))
\end{CD}
\right)
\sr
\left(
\begin{CD}
Z_i'\oo\partial \Delta^i @>>> \Delta(sk_{i-1}(Wr_{r}Cof_c(B')))\\
@VVV @VVV\\
Z_i'\oo\Delta^i @>>> \Delta(sk_i(Wr_{r}Cof_c(B')))
\end{CD}
\right)
$$
All of the objects in these squares are cofibrant and belong to $\Delta^{op}C^{\#}$ and the vertical morphisms are cofibrations. Therefore the diagonals $\Delta(Wr_rCof_c(B))$ and  $\Delta(Wr_rCof_c(B'))$ are in $\Delta^{op}C^{\#}\cap CofEnds$ as well. By Lemma \ref{2009.07.23.2} it remains to show that the morphisms (\ref{skim}) are in $E$.  
The morphisms $Z_i\sr Z_i'$ are in $E$ by property 3, the morphisms $Z_i\oo\partial \Delta^i\sr Z'_i\oo\partial \Delta^i$ and $Z_i\oo\Delta^i\sr Z'_i\oo\Delta^i$ are in $E$ by properties 1,2 and 4 and we conclude that (\ref{skim}) is in $E$ by property 5. 

It remains to show that $E$ is closed under filtered colimits. Let $F,G:I\sr \Delta^{op}Rad(C)$ be two directed diagrams and $f:F\sr G$ a morphism between these diagrams such that for all $i\in I$, $f(i)\in E$. Applying to these diagrams functor $Cof$ or ${\Delta L_*}$ and taking into account that the class of projective equivalences is closed under filtered colimits, we reduce the problem to the case when  the diagrams take values in $\Delta^{op}C^{\#}$. 

The class $E\cap \Delta^{op}C^{\#}$ is $\Delta$-closed by the first part of the proof and by Proposition \ref{2009tdl0} it contains $cl_{\tdl}(\emptyset)$. It is also closed under coproducts by condition 3. Applying to it Proposition \ref{bdlchar} we conclude that it is closed under filtered colimits which finishes the proof of the proposition. 
\end{proof}

\begin{proposition}
\llabel{l7.26.3}
Let $f:X\sr Y$ be a projective cofibration and $i:K\sr L$ a
cofibration of simplicial sets. Then the morphism
$$h(f,i):(X\oo L)\amalg_{X\oo K} (Y\oo K)\sr Y\oo L$$
is a cofibration. If $f$ is a projective equivalence or $i$ is a weak equivalence then $h(f,i)$ is a projective equivalence.
\end{proposition}
\begin{proof}
As was mentioned above, any cofibration is a domain fixing retract of a sequential $I$-cell. For a domain fixing retract $f'$ of $f$ the morphism $h(f',i)$ is a retract of $h(f,i)$ which implies that we may assume that $f$ is a sequential $I$-cell.

To show that $h(f,i)$ is a cofibration it is sufficient to show that the class of morphisms $f$, such
that $h(f,i)$ is a cofibration for all $i$, contains elements of $I$,
i.e. morphisms of the form $U\oo\partial\Delta^n\sr U\oo\Delta^n$, and
is closed under coproducts, push-outs and countable
compositions. For $f$ of the form  $U\oo\partial\Delta^n\sr
U\oo\Delta^n$ the morphism $h(f,i)$ is of the form $U\oo K'\sr U\oo
L'$ where $K'\sr L'$ is a monomorphism of simplicial sets. Any such
morphism is a cofibration because it has the right lifting property
for trivial projective fibrations.  The fact that our class is closed under
coproducts is straightforward. To prove that it is closed under push-outs and countable compositions one has to consider more complex diagrams (especially in the case of push-outs) but the proof remains straightforward. 

Since $f$ is a sequential $I$-cell it is a term-wise coprojection.  If $i$ is a projective  equivalence then Corollary \ref{unexpect} together with Lemma \ref{exactsquares} imply that $h(f,i)\in cl_{\tdl}(\emptyset)$ and therefore is a projective equivalence by Proposition \ref{2009tdl0}.

Assume that $f$ is a projective  equivalence. Then $f\oo Id_K$ and $f\oo Id_L$ are projective equivalences by Corollary \ref{2007ss} and $X\oo L\sr Y\oo K\amalg_{X\oo K} X\oo L$ is a projective equivalence as a push-out of a trivial cofibration. We conclude that $h(f,i)$ is a projective equivalence. 
\end{proof}
\begin{cor}
\llabel{2007c1}
For any cofibrant $X$ and any monomorphism of simplicial sets $K\sr L$ the map $X\oo K\sr X\oo L$ is a projective cofibration. In particular for a cofibrant $X$ and a simplicial set $K$ the object $X\oo K$ is cofibrant.
\end{cor}
Since $Rad(C)$ has colimits we can define the skeletons $sk_n(X)$ of an object $X$ in $\Delta^{op}Rad(C)$ in the usual
way such that $X=colim_n sk_n(X)$ and for each $n$ one has a
push-out square of the form
\begin{equation}
\llabel{eq7.20.1}
\begin{CD}
L_n(X)\oo\Delta^n\amalg_{L_n(X)\oo\partial\Delta^n}
X_n\oo\partial\Delta^n @>>> sk_{n-1}(X)\\
@VVV @VVV\\
X_n\oo\Delta^n @>>> sk_n(X)
\end{CD}
\end{equation}
where $L_n(X)=(sk_{n-1}(X))_n$ is the n-th latching object of $X$. Since the left hand side vertical arrows in (\ref{eq7.20.1}) are as in Proposition \ref{l7.26.3} for $f$ being the morphisms $L_n(X)\sr X_n$ we get the following result.
\begin{cor}
\llabel{2009lcof}
Let $X$ be an object of $\Delta^{op}Rad(C)$ such that the morphisms $L_n(X)\sr X_n$ are projective cofibrations. Then $X$ is projectively cofibrant.
\end{cor}
\begin{theorem}
\llabel{2009issimpl}
The projective closed model structure is simplicial i.e. it satisfies
the axiom SM7 (see e.g. \cite[p.89]{GJ}).
\end{theorem}
\begin{proof}
We have to show that for a projective cofibration $j:A\sr B$ and a
projective fibration $q:X\sr Y$ the morphism
$$S(B,X)\sr S(A,X)\times_{S(A,Y)}S(B,Y)$$
is a Kan fibration which is a weak equivalence if $j$ or $q$ is a weak
equivalence. This follows by adjunction from Proposition \ref{l7.26.3}.
\end{proof}
We let $H(C)$ denote the homotopy category of the projective c.m.s. on $\Delta^{op}Rad(C)$. 
The product $\times$ in $\Delta^{op}Rad(C)$ respects projective equivalences between all objects and defines a product in $H(C)$ which we also denote by $\times$. Let us denote the coproduct in $H(C)$ by $\amalg_{\LL}$ (or, in special cases, $-_{\LL}$ where $-$ is the notation for the coproduct in $\Delta^{op}Rad(C)$) . In general it is computed by the formula
$$X\amalg_{\LL} Y= Cof(X)\amalg Cof(Y)$$
Since the projective c.m.s. is simplicial we also have an adjoint pair of endo-functors $(-)\oo_{\LL} K$ and ${\RR}Hom_{\oo}(K,-)$ on $H(C)$ defined by the formulas
$$X\oo_{\LL} K = Cof(X)\oo K$$
$${\RR}Hom_{\oo}(K,X)=Hom_{\oo}(K, Fib_{proj}(X))$$
\begin{proposition}
\llabel{2009aoh}
For any $C$ one has:
\begin{enumerate}
\item if $X$, $Y$ are in $\Delta^{op}C^{\#}$ then $X\amalg_{\LL} Y=X\amalg Y$,
\item if $X$ is in $\Delta^{op}C^{\#}$ and $K$ is a simplicial set then $X\oo_{\LL} K=X\oo K$.
\end{enumerate}
\end{proposition}
\begin{proof}
It follows easily from Theorem \ref{coin}. 
\end{proof}
\begin{proposition}
\llabel{2007lp}
The following conditions on $C$ are equivalent:
\begin{enumerate}
\item the projective closed model structure on $\Delta^{op}Rad(C)$ is left proper,
\item for any $f\in W_{proj}$ and $Z\in C$ one has $f\amalg Id_Z\in W_{proj}$, 
\item for any $f\in W_{proj}$ and $Z\in \Delta^{op}C^{\#}$ one has $f\amalg Id_Z\in W_{proj}$, 
\item for any $X\in \Delta^{op}Rad(C)$ and $Z\in \Delta^{op}C^{\#}$ one has $X\amalg_{\LL} Z=X\amalg Z$.
\end{enumerate}
\end{proposition}
\begin{proof}
Recall that a c.m.s. is called left proper if for any push-out square
\begin{eq}
\llabel{2009lp}
\begin{CD}
A @>f>> X\\
@VgVV @VVg'V\\
B @>>> Y
\end{CD}
\end{eq}
where $f$ is a cofibration and $g$ is a weak equivalence, $g'$ is a weak equivalence. Since objects of $C$ are cofibrant this immediately implies the (1)$\Rightarrow$(2) part of the lemma. 

To prove that (2)$\Rightarrow$(1) assume that for any $f\in W_{proj}$ and $Z\in C$ one has $f\amalg Id_Z\in W_{proj}$. Then since $W_{proj}$ is closed under filtered colimits the same holds for $Z\in C^{\#}$. Consider a square of the form (\ref{2009lp}).

By Proposition \ref{2009cofrep}, $f$ is a domain preserving retract of a morphism $f'$ whose terms are of the form  $A_n\sr A_n\amalg Z_n$ where $Z_n\in \bar{C}$.  Then $g'$ is a retract of the push-out of $g$ by $f'$ and since $W_{proj}$ is closed under retracts we may assume that terms of $f$ are of the form $A_n\sr A_n\amalg Z_n$ for $Z_n\in \bar{C}$.   Then $g'$ is the diagonal of a morphism whose rows are of the  form $g\amalg Id_{Z_n}$. This finishes the proof by Proposition \ref{2009tdl0}.

The implication (3)$\Rightarrow$(2) is obvious. The implication (2)$\Rightarrow$(3) follows easily from Proposition \ref{2009tdl0}. The equivalence between (3) and (4) is obvious.
\end{proof}
\begin{proposition}
\llabel{2009neq}
The following conditions on $C$ are equivalent:
\begin{enumerate}
\item for any $f\in W_{proj}$ and $Z\in Rad(C)$ one has $f\amalg Id_Z\in W_{proj}$,
\item the class $W_{proj}$ is $(\tdl,\amalg)$-closed,
\item for any $X$, $Y$ in $\Delta^{op}Rad(C)$ one has $X\amalg_{\LL} Y=X\amalg Y$.
\end{enumerate}
\end{proposition}
\begin{proof}
The implication (1)$\Rightarrow$(2) follows immediately from Proposition \ref{2009tdl0}. The inverse implication and the equivalence between (2) and (3) are obvious. 
\end{proof}
\begin{examples}\rm\llabel{2009exs2}
\begin{enumerate}
\item An example of a category $C$ such that the projective c.m.s. is not left proper is given in \ref{2009bip}.
\item For any commutative ring $A$ the category of finitely generated free commutative algebras over $A$ satisfy the conditions of Proposition \ref{2007lp}.
\item The category of finitely generated free commutative algebras over $A$ satisfies the conditions of  Proposition \ref{2009neq} if and only if $A$ is a field. 
\end{enumerate}
\end{examples}

Let now $C$ be a pointed category i.e. the initial object of $C$ is also a final object. We will denote this object by $pt$ and the coproduct by $\vee$. Note that $pt$ is also the initial and final object in $Rad(C)$. As usual we will write $X/Y$ for the pushout of the pair of morphisms $Y\sr X$, $Y\sr pt$.  Since $C$ is pointed the category, $H(C)$ is equipped with a functor $\Sigma$ which take values in co-monoids over $H(C)$ with respect to $\vee_{\LL}$ (see \cite[I.2]{Quillen} or \cite[Ch. 6]{Hovey}) and with the class of cofibration "sequences".  

For a pointed simplicial set $K$ and $X\in \Delta^{op}Rad(C)$ denote by $X\wedge K$ the object defined by the elementary push-out square
$$
\begin{CD}
X @>>> X\oo K\\
@VVV @VVV\\
pt @>>> X\wedge K
\end{CD}
$$
in $\Delta^{op}Rad(C)$, where the upper arrow is defined by the distinguished point of $K$.  
\begin{lemma}
\llabel{2009lc1}
The suspension functor $\Sigma^1:H(C)\sr H(C)$ (see \cite[2.9]{Quillen}) may be represented on $\Delta^{op}C^{\#}$ by $X\mapsto X\wedge S^1$ where $S^1=\Delta^1/\partial\Delta^1$ is the simplicial circle. 
\end{lemma}
\begin{proof}
Since the projective c.m.s. is simplicial the suspension functor $\Sigma^1$ on $H(C)$ may be defined as $X\mapsto X\wedge S^1$ on cofibrant objects $X$. For any $X$ the standard cofibrant replacement $Cof(X)$ of $X$ lies in $\Delta^{op}C^{\#}$ and for $X\in \Delta^{op}C^{\#}$ the morphism $p:Cof(X)\sr X$ is in $cl_{\tdl}(\emptyset)$ by Theorem \ref{coin}. Therefore, $p\wedge S^1$ is in $cl_{\tdl}(\emptyset)$ by Lemma \ref{fandt} and by Theorem \ref{coin} it is a projective equivalence.
\end{proof}
For a morphism $f:X\sr X'$ define $cone(f)$ by the elementary push-out square (in $\Delta^{op}Rad(C)$)
$$
\begin{CD}
X @>>> X\wedge\Delta^1\\
@VVV @VVV\\
X' @>>> cone(f)
\end{CD}
$$
\begin{definition}
\llabel{2009cs}
A (term-wise) coprojection sequence in $\Delta^{op}Rad(C)$ is a pair of morphisms 
$$X\sr Y\sr Z$$
such that the first morphism is a term-wise coprojection and $Y\sr Z$ is isomorphic to $Y\sr Y/X$. 
\end{definition}
Clearly, any term-wise coprojection $X\sr Y$ extends canonically to a term-wise coprojection sequence $X\sr Y\sr Y/X$.
\begin{lemma}
\llabel{2007.1}
For any term-wise coprojection sequence $X\stackrel{f}{\sr} Y\sr Z$ there is a commutative diagram 
$$\begin{CD}
X @>>> cyl(f) @>>> cone(f)\\
@VVV @V1VV @V2VV\\
X@>>> Y@>>> Z
\end{CD}
$$
where the arrows (1) and (2) are in $cl_{\Delta}(\emptyset)$ and therefore in $W_{proj}$. 
\end{lemma}
\begin{proof}
By  Proposition \ref{2009tdl0} it is sufficient to verify that morphisms (1) and (2) are in $cl_{\Delta}(\emptyset)$.  Our diagram is a part of the morphism of elementary push-out squares $Q_1\sr Q_2$ where
$$
Q_1=
\left(
\begin{CD}
X @>>> Cyl(f)\\
@VVV @VVV\\
pt @>>> cone(f)
\end{CD}
\right)
\,\,\,\,\,\,\,\,\,\,
Q_2=
\left(
\begin{CD}
X @>>> Y\\
@VVV @VVV\\
pt @>>> Z
\end{CD}
\right)
$$
The morphism $Cyl(f)\sr Y$ is in $cl_{\Delta,\amalg_{<\infty}}(\emptyset)$ by Lemma \ref{cyl} and therefore the morphism $cone(f)\sr Z$ is in  $cl_{\Delta,\amalg_{<\infty}}(\emptyset)$ by Lemma \ref{lm1}, \ref{lm0} and the 2-out-of-3 property of $\Delta$-closure. By Proposition \ref{fcop} we have $cl_{\Delta, \amalg_{<\infty}}(\emptyset)=cl_{\Delta}(\emptyset)$.
\end{proof}
For any $f:X\sr X'$ one has 
\begin{eq}
\llabel{2009cone}
cone(f)/X=(X'/X)\vee (X\wedge S^1)
\end{eq}
Let $X\stackrel{f}{\sr} Y{\sr} Z$ be a term-wise coprojection sequence. Then $cone(f)\sr Z$ is a projective equivalence and (\ref{2009cone}) defines a morphism $Z\sr Z\vee (X\wedge S^1)$ in $H(C)$. If $Z$ and $X$ are in $\Delta^{op}C^{\#}$ then $Z\vee (X\wedge S^1)$ is canonically isomorphic to $Z\vee_{\LL} \Sigma^1(X)$ by Lemma \ref{2009lc1}.  In particular any term-wise coprojection sequence in $\Delta^{op}C^{\#}$ defines in a natural way a pair of the form
\begin{eq}
\llabel{2009cfs0}
(X\sr Y\sr Z,\,\,\, Z\sr Z\vee_{\LL} \Sigma^1(X))
\end{eq}
in $H(C)$.
\begin{theorem}
\llabel{2009cfs}
A pair of the form
\begin{eq}
\llabel{2009cfs1}
(X'\sr Y'\sr Z',\,\,\, Z'\sr Z'\vee_{\LL} \Sigma^1(X'))
\end{eq}
in $H(C)$ is a cofibration sequence (see \cite[Def. 6.2.1, p.156]{Hovey})  if it is isomorphic in $H(C)$ to a pair (\ref{2009cfs0}) for a term-wise coprojection sequence in $\Delta^{op}C^{\#}$ and only if it is isomorphic to such a pair for a term-wise coprojection sequence in $\Delta^{op}\bar{C}$.
\end{theorem}
\begin{proof}
"if" Using the standard decomposition of the morphism $Cof(X)\sr X\sr Y$ into a cofibration and a trivial fibration, we get a square
$$
\begin{CD}
Cof(X) @>\tilde{f}>> \tilde{Y}\\
@VVV @VVV\\
X @>f>> Y
\end{CD}
$$
where the vertical arrows are projective equivalences, the upper horizontal one is a cofibration and a term-wise coprojection and all objects are in $\Delta^{op}C^{\#}$. It remains to show that the pair
$$(Cof(X)\sr \tilde{Y}\sr \tilde{Y}/Cof(X),\,\, \tilde{Y}/Cof(X)\sr (\tilde{Y}/Cof(X))\vee_{\LL} \Sigma^1(Cof(X)))$$
is isomorphic in $H(C)$ to (\ref{2009cfs1}). This follows easily from Lemmas \ref{lm1}, \ref{lm0} and Theorem \ref{coin}.

"only if" This direction follows easily by an argument similar to the proof of "if".
\end{proof}

\subsection{$E$-local equivalences}
\label{easysec}
The following results and definitions concerning $E$-local objects and $E$-local equivalences are mostly standard. In particular our definitions agree with the ones given in \cite{Hirs}. 
\begin{definition}
\llabel{2007locdef}
Let $E$ be a class of morphisms in $\Delta^{op}Rad(C)$. An object $Y$ of
this category is called $E$-local if it is projectively fibrant and for any simplicial set $K$
and any element $f:X\sr X'$ in $E$ the map
\begin{eq}
\llabel{hc1}
Hom_{H(C)}(X',Hom_{\oo}(K,Y))\sr Hom_{H(C)}(X,Hom_{\oo}(K,Y))
\end{eq}
defined by $f$, is bijective. 
\end{definition}
\begin{lemma}
\llabel{d7.20.4}
Let $E$ be a class of morphisms in $\Delta^{op}Rad(C)$. An object $Y$ of
this category is $E$-local iff it is projectively fibrant and
for any element $f:X\sr X'$ of $E$  and a representative $\tilde{f}:\tilde{X}\sr \tilde{X'}$  of $f$ such that $\tilde{X}$ and $\tilde{X}'$ are cofibrant, the map of simplicial sets 
$$S(\tilde{X}', Y)\sr S(\tilde{X}, Y)$$
defined by $\tilde{f}$, is a weak equivalence. 
\end{lemma}
\begin{proof}
Note first that $S(\tilde{X}', Y)$ and $S(\tilde{X}, Y)$ are Kan simplicial sets.  A map between two such sets is a weak equivalence iff for any $K$ it induces the bijection on homotopy classes of maps from $K$. These homotopy classes of maps are identified with the sides of (\ref{hc1}) by \cite[Prop. 3.10, p.93]{GJ}. 
\end{proof}
\begin{definition}
\llabel{2007eq}
Let $E$ be a class of morphisms in $\Delta^{op}Rad(C)$. A morphism
$f:X\sr X'$ is called a (left) E-local equivalence if for any
$E$-local $Y$ the map
$$Hom_{H(C)}(X',Y)\sr Hom_{H(C)}(X,Y)$$
defined by ${f}$ is bijective. 
\end{definition}
\begin{lemma}
\llabel{d7.20.5}
Let $E$ be a class of morphisms in $\Delta^{op}Rad(C)$. A morphism
$f:X\sr X'$ between cofibrant objects is an $E$-local equivalence  if and only if for any
$E$-local $Y$ the map of simplicial sets 
$$S(X', Y)\sr S(X, Y)$$
defined by $f$ is a weak equivalence. 
\end{lemma}
\begin{proof}
Similar to the proof of Lemma \ref{d7.20.4}.
\end{proof}
We denote the class of $E$-local equivalences by $cl_{l}(E)$.  Smith's localization theorem together with Proposition \ref{2009pr1} gives the following result.
\begin{theorem}
\llabel{2009th1}
Assume that the projective c.m.s. on $\Delta^{op}Rad(C)$ is left proper. Then for any set of morphisms $E$ in $\Delta^{op}Rad(C)$ there exists a closed model structure on $\Delta^{op}Rad(C)$ with the classes of weak equivalences and cofibrations being $cl_l(E)$ and $Cof$ respectively. 
\end{theorem}
\begin{proposition}
\llabel{2009afg}
Assume that the projective c.m.s. on $\Delta^{op}Rad(C)$ is left proper and let $E$ be such that the domains and codomains of its elements are cofibrant and compact. Then the $E$-local c.m.s is almost finitely generated (see \cite[Def. 4.1, p.82]{Hoveystab}).
\end{proposition}
\begin{proof}
It follows from \cite[Prop. 4.2, p.83]{Hoveystab} and Lemma \ref{2009comcom}.
\end{proof}
Example \ref{2009bip} shows that the projective c.m.s. is not always left proper and that the pair $(cl_l(E),Cof)$ does not always define a c.m.s. Nevertheless a number of important properties of $cl_l(E)$ can be proved for any $C$ and $E$ as will be shown below. 
\begin{example}\rm
\llabel{2009bip}
Let $C$ be the category $\{1,2\}/Fsets$ whose objects are maps $\{1,2\}\sr X$ where $X$ is a finite set and morphisms are the obvious commutative triangles. This category has the initial object $I=\{1,2\}$ and finite coproducts given by push-out squares of sets under $\{1,2\}$ which we will denote by $\vee$ by analogy with the wedge of pointed sets. 

Let $pt=(\{1,2\}\sr \{1\})$ be the final object of $C$ and $X=\{1,2\}\subset \{1,2,3\}$. Then any object of $C$ is a coproduct of the form $(\vee_n X)\vee (\vee_m pt)$ where $n\ge 0$ and $m=0,1$. Therefore, a radditive functor $F$ on $C$ is determined by its values on $X$ and $pt$ with $F(pt)$ being $pt$ or $\emptyset$. Two morphisms from $X$ to $I$ define a map $\{1,2\}\sr F(X)$ and explicit considerations show that $Rad(C)$ can be described  as the subcategory of the category of pairs $((X;x_1,x_2),\phi)$ where $(X;x_1,x_2)$ is a bi-pointed set and $\phi=pt$ or $\phi=\emptyset$ which consists of pairs such that $x_1=x_2$ or $\phi=\emptyset$. 

The functor represented by a  finite by-pointed set $(X;x_1,x_2)$ is $((X,x_1,x_2),\emptyset)$ if $x_1\ne x_2$ and $((X,x_1,x_2),pt)$ if $x_1=x_2$. Therefore all radditive functors except for the ones of the form $((X;x,x),\emptyset)$ are in $C^{\#}$. Any monomorphism in $C^{\#}$ is a coprojection and in particular a projective cofibration.

However the canonical morphism $I\sr pt$ is not a mono since there are two morphisms $X\sr I$, but it is a projective cofibration as the morphism from the initial object to a representable functor. Therefore not all projective cofibrations in $C$ are monomorphisms. 

The coproduct of two radditive functors is given by 
$$((X;x_1,x_2),\phi)\vee ((X';x'_1,x'_2),\phi')=((X;x_1,x_2)\vee(X';x_1',x_2'),\phi\cup\phi')$$
A simplicial radditive functor on $C$ is a pair $((X;x_1,x_2),\phi)$ where now $(X;x_1,x_2)$ is a bi-pointed simplicial set and $\phi=\emptyset$ or $\phi=pt$ with the condition that if $x_1\ne x_2$ then $\phi=\emptyset$. 

Let $\Psi=((pt;pt,pt),\emptyset)$ be the image of the canonical map $I\sr pt$ and $Y=((\Delta^1;x_0,x_1),\emptyset)$ where $x_0$, $x_1$ are two vertices of $\Delta^1$. The morphism $Y\sr \Psi$ is a projective equivalence. Using Corollary \ref{2009lcof} one verifies easily that $Y$ is cofibrant. On the other hand 
$$Y\vee pt=((S^1;x,x),pt)\sr pt=\Psi\vee pt$$
is not a projective equivalence which shows that the projective c.m.s on $\Delta^{op}Rad(C)$ is not left proper. 

Let $E=\{f:I\sr pt\}$. Note that $f$ is a cofibration between cofibrant objects. An object $((X,x_1,x_2),\phi)$ is $E$-local if and only if $\phi=pt$ and $x_1=x_2$ and one verifies easily that $H(C,E)$ is equivalent to the homotopy category of pointed simplicial sets. 

On the other hand we have a push-out square
$$
\begin{CD}
I @>>> \Psi\\
@VfVV @VVV\\
pt @>>> pt
\end{CD}
$$
and since $\Psi$ is isomorphic to $Y\vee pt=(S^1,pt)$  in $H(C,E)$ the morphism $\Psi\sr pt$ is not an $E$-local equivalence. We conclude that the class $cl_l(E)\cap Cof$ is not closed under push-outs and therefore the left Bousfield localization of the projective c.m.s. by $E$ does not exist. 
\end{example}

\begin{theorem}
\llabel{2009innov2}
For any $E$ in $\Delta^{op}Rad(C)$ the class $cl_l(E)$ is $\tdl$-closed.
\end{theorem}
\begin{proof}
We apply to $cl_l(E)$ Proposition \ref{2009.07.23.1}. The first two conditions are obvious. 
Coproducts of $E$-local equivalences between cofibrant objects are $E$-local equivalences which implies condition (3). Condition (4) follows easily from Lemma \ref{d7.20.5}.
The same lemma together with the fact that a morphism of pull-back squares of simplicial sets where vertical arrows are fibrations and which is a weak equivalence on three generating vertices is also a weak equivalence on the fourth vertex, implies condition (5). 
\end{proof}

\begin{theorem}
\llabel{exlemma} 
For any set of morphisms $E$ in $\Delta^{op}Rad(C)$ there exists a functor 
$$Ex_E:\Delta^{op}Rad(C)\sr \Delta^{op}C^{\#}$$
and a natural transformation $Id\sr Ex_E$ such that one has:
\begin{enumerate}
\item\llabel{exlemma1} for any $X$ the object $Ex_E(X)$ is $E$-local,
\item\llabel{exlemma2} for $X$ in $\Delta^{op}C^{\#}$ the morphism $X\sr Ex_E(X)$ is in $cl_{\tdl}(Cof(E)\amalg Id_C)$,
\item\llabel{exlemma3} for $X$ in $\Delta^{op}C^{\#}$ the morphism $X\sr Ex_E(X)$ is in $cl_{l}(E)$.
\end{enumerate}
\end{theorem}
\begin{proof}
For a morphism $f:X\sr X'$ denote by $i_f:X\sr cyl(f)$ the composition
$$X\stackrel{Id\oo\partial_1}{\sr}X\oo\Delta^1\sr cyl(f).$$
For any $f$,  $i_f$ is a term-wise coprojection which is homotopy equivalent to $f$ by  Lemma \ref{cyl}. Let further $A_E$ be the class of morphisms 
of the form
$$(cyl(f)\oo\partial\Delta^i)\amalg_{X\oo\partial\Delta^i}(X\oo\Delta^i)\sr cyl(f)\oo\Delta^i$$
defined by $i_f$ for $f:X\sr X'$ in $Cof(E)$. 

Let $Y$ be an object such that the morphism $Y\sr pt$ has the right lifting property for $A_E\cup J_{proj}$. Then $Y$ is projectively cofibrant and for any $f:X\sr X'$ in $Cof(E)$ the map $S(cyl(f),Y)\sr S(X,Y)$ defined by $f$ is a trivial Kan fibration. Since $cyl(f)$ is homotopy equivalent to $X'$ this implies that the map $S(X',Y)\sr S(X,Y)$ defined by $f$ is a weak equivalence and we conclude by Lemma \ref{d7.20.4} that $Y$ is $E$-local. 

Therefore, by Proposition \ref{2009pr1} and \cite[Prop. 1.3, p.452]{Beke} we can use the transfinite small object argument to construct for any $X$ a functorial decomposition of the morphism $X\sr pt$ of the form $X\sr Ex_E(X)\sr pt$ where $X\sr Ex_E(X)$ is an $(A_E\cup J_{proj})$-cell and $Ex_E(X)$ is $E$-local. 

Elements of $A_E$ and $J_{proj}$ are term-wise coprojections between objects of $\Delta^{op}C^{\#}$. Applying Lemma \ref{exactsquares} and the 2-out-of-3 property of $\tdl$-closed classes we conclude that any $(A_E\cup J_{proj})$-cell which starts at an object of $\Delta^{op}C^{\#}$ lies in $cl_{\tdl}(Cof(E)\amalg C^{\#})=cl_{\tdl}(Cof(E)\amalg Id_C)$. 

By Proposition \ref{2009aoh}(1) we conclude that $Cof(E)\amalg Id_C\subset cl_l(E)$ and therefore by Theorem \ref{2009innov2}, that $cl_{\tdl}(Cof(E)\amalg Id_C)\subset cl_l(E)$. The theorem is proved. 
\end{proof}
\begin{theorem}
\llabel{l7.19.9.1}
\llabel{2009main}
Let $E$ be a set of morphisms in 
$\Delta^{op}C^{\#}$. Then one has
$$cl_l(E)\cap \Delta^{op}C^{\#}=cl_{\tdl}(E\amalg Id_C)$$
where the closure on the right is in $\Delta^{op}C^{\#}$.
\end{theorem}
\begin{proof}
"$\subset$" Let $f:X\sr Y$ be an element of $cl_l(E)\cap \Delta^{op}C^{\#}$. Consider the diagram 
$$
\begin{CD}
X @>>> Ex_E(X)\\
@VVV @VVV\\
Y @>>> Ex_E(Y)
\end{CD}
$$
The horizontal arrows are in $cl_l(E)$ by  Theorem \ref{exlemma}(\ref{exlemma3}).  Therefore the right hand side vertical arrow is in $cl_l(E)$. Since the objects in question are $E$-local by Theorem \ref{exlemma}(\ref{exlemma1}) the right vertical arrow is a projective equivalence, and by Theorem \ref{coin} it is an element of $cl_{\tdl}(\emptyset)$. By Theorem \ref{exlemma}(\ref{exlemma2}) the horizontal arrows are in $cl_{\tdl}(E\amalg Id_C)$.

"$\supset$"  By Theorem \ref{2009innov2} we have $cl_{\tdl}(E\amalg Id_C)\subset cl_l(E\amalg Id_C)$ and Proposition \ref{2009aoh} implies that $cl_l(E\amalg Id_C)=cl_{l}(E)$.
\end{proof}
\begin{cor}
\llabel{2009main2}
For any set of morphisms $E$ in $\Delta^{op}C^{\#}$ one has
\begin{eq}
\llabel{2009gd1}
cl_l(E)=cl_{\tdl}((E\amalg Id_C)\cup W_{proj})
\end{eq}
\end{cor}
\begin{proof}
Let $f:X\sr Y$ be in $cl_l(E)$. Taking the standard cofibrant replacement of $f$ and using the 2-out-of-3 property for $cl_{\Delta}$ we may assume that $X$ and $Y$ are in $\Delta^{op}C^{\#}$. Then $f\in cl_{\tdl}(E\amalg Id_C)$. The opposite inclusion immediately follows  from Theorem \ref{2009main}.
\end{proof}
\begin{cor}
\llabel{2009cs1}
Let $C$ be a pointed category and $(X\stackrel{f}{\sr} Y\stackrel{g}{\sr} Z,\,\,\, Z\sr Z\vee_{\LL} \Sigma^1(X))$ a cofibration sequence in $H(C)$. Then one has:
$$g\in cl_{l}(\{X\sr pt\})$$
$$(pt\sr Z)\in cl_{l}(\{f\})$$
\end{cor}
\begin{proof}
By Theorem \ref{2009cfs} we may assume that there is an elementary push-out square
$$
\begin{CD}
X @>f>> Y\\
@VVV @VVgV\\
pt @>>> Z
\end{CD}
$$
in $\Delta^{op}C^{\#}$. By Lemma \ref{exactsquares} we conclude that $g\in cl_{\Delta,\amalg_{<\infty}}(\{X\sr pt\})$ and $(pt\sr Z)\in cl_{\Delta,\amalg_{<\infty}}(\{f\})$. Applying Theorem \ref{2009main} and Proposition \ref{2009fcop} we get the conclusion of the proposition.
\end{proof}

\section{Functoriality results}
Let $C,C'$ be two categories with finite coproducts and $F:C\sr Rad(C')$ be a functor. Considering $Rad(C')$ as a full subcategory of $Funct((C')^{op},Sets)$ we get a pair of adjoint functors $F^*,F_*$ where $F^*:Funct(C^{op},Sets)\sr Funct((C')^{op},Sets)$ is the left adjoint which extends $F$. As a left adjoint it commutes with colimits and since it takes representable functors to radditive ones we conclude by Lemma \ref{2007s1} that it defines a functor
$$F^{rad}:Rad(C)\sr Rad(C')$$
\begin{lemma}
\llabel{2007f3}
For any $F$ the functor $F^{rad}$ commutes with filtered colimits and reflexive coequalizers.
\end{lemma}
\begin{proof}
It follows from Proposition \ref{2009elprop}(\ref{rad5} ,\ref{2007rad1}).
\end{proof} 
We will use the notation $F^{rad}$ also in the case when $F$ is a functor from $C$ to $C'$ considering $C'$ as a full subcategory of $Rad(C')$ or from $C$ to $(C')^{\#}$ or, equivalently, when $F$ is a continuous functor from $C^{\#}$ to $(C')^{\#}$.  Lemma \ref{2007f3} implies in particular that for any $F:C\sr C'$ the functor $F^{rad}$ takes $C^{\#}$ to $(C')^{\#}$ and that the resulting functor $C^{\#}\sr (C')^{\#}$ is continuous. One observes easily that this construction commutes with compositions of functors. 
\begin{example}
\llabel{exsp2}\rm
Let $C$ be the category of  pointed finite simplicial sets, $D$ the category of free finite simplicial sets and $F:C\sr D$ the forgetful functor. Then one can easily see that $F^{rad}$ does not take $\bar{C}$ to $\bar{D}$ since an infinite wedge of finite simplicial sets can not be represented as an infinite coproduct of finite simplicial sets. 
\end{example}
\begin{lemma}
\llabel{2009lcomp}
Let $C^{\#}\stackrel{F}{\sr} (C')^{\#}\stackrel{G}{\sr} (C'')^{\#}$ be a composable pair of continuous functors (resp. a composable pair of functors $C\stackrel{F}{\sr} C'\stackrel{G}{\sr} C''$). Then there is a canonical isomorphism $(F\circ G)^{rad}=F^{rad}\circ G^{rad}$.
\end{lemma}
\begin{proof}
It follows from Lemma \ref{2007f3} since any radditive functor on $C$ is a reflexive coequalizer of a diagram in $C^{\#}$.
\end{proof}
\begin{lemma}
\llabel{2007f2}
Let  $F$ be a functor $C\sr Rad(C')$. Then the functor $F^{rad}$ has a right adjoint $F_{rad}$ if and only if $F$ respects finite coproducts. In that case the right adjoint maps $X\in Rad(C')$ to the functor given by 
\begin{eq}
\llabel{2009radj}
U\mapsto Hom_{Funct((C')^{op},Sets)}(F(U),X)
\end{eq}
\end{lemma}
\begin{proof}
"If" The functor $F^*:Funct(C^{op},Sets)\sr Funct((C')^{op},Sets)$ always has a right adjoint $F_*$ given by (\ref{2009radj}) and if $F$ commutes with finite coproducts then $F_*$ takes radditive functors to radditive functors and defines a right adjoint to $F^{rad}$. 

"Only if" It follows from the fact that a left adjoint preserves colimits and therefore if $F^{rad}$ has a right adjoint, $F$ must commute with finite coproducts by Proposition \ref{2009elprop}(\ref{rad6}).
\end{proof}
\begin{lemma}
\llabel{2009frad1}
Let $F:C\sr C'$ be a functor which commutes with finite coproducts. Then $F_{rad}$ commutes with filtered colimits and reflexive coequalizers.
\end{lemma}
\begin{proof}
The functor $F_{rad}$ is the restriction to $Rad(C')$ of the functor $F_*$ which commutes with all colimits. The inclusion $Rad(C')\sr Funct((C')^{op},Sets)$ commutes with filtered colimits and reflexive coequalizers by Proposition \ref{2009elprop}(\ref{rad5},\ref{2007rad1}). Therefore $F_{rad}$ also commutes with these two types of colimits.
\end{proof}
Note that the conclusion of Lemma \ref{2009frad1} may be false for continuous functors $C^{\#}\sr (C')^{\#}$. 

Let $F:C\sr C'$ be a functor which commutes with finite coproducts and which is surjective on objects. Then we get an adjoint pair of functors $(F^{rad},F_{rad})$ between $Rad(C)$ and $Rad(C')$ such that $F_{rad}$ reflects isomorphisms. By Lemma \ref{2009frad1}, $F_{rad}$ commutes with reflexive coequalizers. Therefore, by Beck's Theorem (see \cite{Beck}) in the reflexive coequalizer form we get the following result.
\begin{proposition}
\llabel{2009mon}
Under the assumptions made above the category $Rad(C')$ is equivalent to the category of $M$-algebras, where $M=F^{rad}F_{rad}$ is the monad (triple) defined by the adjoint pair $(F^{rad},F_{rad})$. 
\end{proposition}
\begin{cor}
\llabel{2009mon1}
A category is equivalent to the category of radditive functors if and only if it is equivalent to the category of algebras over a continuous  monad $M$ on the category  $Sets^A$ of families of sets.
\end{cor}
This is a reformulation in the language of radditive functors of \cite[Th. 3.18, p.149]{Adamek} where continuous functors are called finitary functors.

\begin{theorem}
\llabel{2009th2} Let $C$, $D$ be categories with finite coproducts and $F:C\sr Rad(C')$ a functor. Then $F^{rad}:\Delta^{op}C^{\#}\sr \Delta^{op}Rad(C')$ takes projective equivalences to projective equivalences.
\end{theorem}
\begin{proof}
It follows from Theorem \ref{coin}, Lemma \ref{fandt} and Proposition \ref{2009tdl0}. 
\end{proof}
\begin{cor}
\llabel{2009forin}
Let $C$, $C'$ be categories with finite coproducts and let $F:C^{\#}\sr (C')^{\#}$ be a continuous functor. Then the simplicial extension $F:\Delta^{op}C^{\#}\sr \Delta^{op}(C')^{\#}$ of $F$ takes projective equivalences to projective equivalences.
\end{cor}

For a functor $F$ denote by $iso(F)$ the class of morphisms $f$
such that $F(f)$ is an isomorphism. Recall that a functor is called a strict
localization if it is a localization and any morphism in the target category is isomorphic to the image of a morphism in the source category
\begin{proposition}
\llabel{preloc} Let $C$ be a small category with finite
coproducts. Then the functor $\Phi:\Delta^{op}C^{\#}\sr H(C)$ is a strict
localization and $iso(\Phi)=cl_{\tdl}(\emptyset)$.
\end{proposition}
\begin{proof}
The fact that $\Phi$ is a strict localization follows immediately from the fact that the standard fibrant-cofibrant replacement of any object of $\Delta^{op}Rad(C)$ belongs to $\Delta^{op}C^{\#}$. 
The fact that $iso(\Phi)=cl_{\tdl}(\emptyset)$ follows from Theorem \ref{coin}.
\end{proof}
\begin{remark}
\rm Proposition \ref{preloc} remains valid if we replace $C^{\#}$ with $\bar{C}$.
\end{remark}
In view of Proposition \ref{preloc} and Theorem \ref{2009th2} any functor $F:C\sr Rad(C')$ defines a functor $\LL F^{rad}:H(C)\sr H(C')$. One verifies immediately that for any composable pair of continuous functors $F:C^{\#}\sr (C')^{\#}$ and $G:(C')^{\#}\sr (C'')^{\#}$ there is a canonical isomorphism 
$$\LL (G\circ F)^{rad}=\LL G^{rad}\circ \LL F^{rad}.$$
\begin{lemma}
\llabel{2007fcc}
Let $F:C\sr C'$ be a functor which commutes with finite coproducts. Then $F_{rad}$ takes projective equivalences between objects of $\Delta^{op}Rad(C')$ to projective equivalences and the corresponding functor ${\RR}F_{rad}$ is right adjoint to $\LL F^{rad}$. 
\end{lemma}
\begin{proof}
The fact that $F_{rad}$ takes projective equivalences to projective equivalences is obvious from the definitions. To prove that ${\RR}F_{rad}$ and ${\bf L}F^{rad}$ are adjoint it is sufficient to construct natural transformations $Id\sr {\RR}F_{rad}\LL F^{rad}$ and $\LL F^{rad}{\RR}F_{rad}\sr Id$ such that the compositions
$${\RR}F_{rad}\sr {\RR}F_{rad}\LL F^{rad}{\RR}F_{rad}\sr {\RR}F_{rad}$$
$$\LL F^{rad}\sr \LL F^{rad}{\RR}F_{rad}\LL F^{rad}\sr \LL F^{rad}$$
are identities. Note first that $\LL F^{rad}$ fits into a commutative square
$$
\begin{CD}
\Delta^{op}Rad(C) @>F^{rad}\circ {\Delta L_*}>> \Delta^{op}Rad(C')\\
@VVV @VVV\\
H(C) @>\LL F^{rad}>> H(C')
\end{CD}
$$
where $L_*$ is the resolution functor defined in Section \ref{pcms}.

Let $a:F^{rad}F_{rad}\sr Id$ and $Id\sr F^{rad}F_{rad}$ be the adjunctions between $F_{rad}$ and $F^{rad}$. Define the adjunctions between ${\RR}F_{rad}$ and $\LL F^{rad}$ as follows. For $\LL F^{rad} {\RR}F_{rad}\sr Id$ we take
$$F^{rad}{\Delta L_*} F_{rad}\sr F^{rad}F_{rad}\stackrel{a}{\sr} Id$$
and for $Id\sr {\RR}F_{rad}\LL F^{rad}$ we take
$$Id \sr {\Delta L_*}\stackrel{b}{\sr} F_{rad} F^{rad} {\Delta L_*}$$
where the first arrow is the inverse in $H(C)$  of the morphism ${\Delta L_*}\sr Id$. That the first composition is identity follows from the diagram
$$
\begin{CD}
{\Delta L_*} F_{rad} @>>> F_{rad}  F^{rad} {\Delta L_*} F_{rad} @>>> F_{rad} \\
@VVV @VVV @VVV\\
F_{rad} @>>> F_{rad} F^{rad} F_{rad} @>>> F_{rad}
\end{CD}
$$
and the fact that the pair $(a,b)$ is an adjunction between $F_{rad}$ and $F^{rad}$. For the second composition consider the diagram
$$
\begin{CD}
F^{rad}{\Delta L_*} Id @<<< F^{rad}{\Delta L_*}{\Delta L_*} @>>> F^{rad}{\Delta L_*}F_{rad} F^{rad}{\Delta L_*} @>>> F^{rad}{\Delta L_*}\\
@. @VVV @VVV @VVV\\
@. F^{rad} Id {\Delta L_*} @>>> F^{rad} F_{rad} F^{rad}{\Delta L_*} @>>> F^{rad}{\Delta L_*}
\end{CD}
$$
The lower composition is the identity since $(a,b)$ is an adjunction. To check that the upper one is the identity in $H(C')$ it remains to verify that the two morphisms
$$F^{rad}{\Delta L_*}{\Delta L_*}\sr F^{rad}Id {\Delta L_*}$$
$$F^{rad}{\Delta L_*} {\Delta L_*}\sr F^{rad} {\Delta L_*} Id$$
coincide in $H(C')$. Since all the functors involved respect projective equivalences it is sufficient to check it for $X\in \Delta^{op}C^{\#}$. For such an $X$ it follows from the commutative square
$$
\begin{CD}
{\Delta L_*}{\Delta L_*}(X) @>>> {\Delta L_*} Id(X)\\
@VVV @VVV\\
Id {\Delta L_*}(X) @>>> X
\end{CD}
$$
\end{proof}
\begin{remark}\rm
It is obvious from definitions that in the context of Lemma \ref{2007fcc} the pair $(F^{rad},F_{rad})$ forms a Quillen adjunction and our functors $\LL F^{rad}$, $\RR F_{rad}$ are canonically isomorphic to the standard derived functors for this adjunction.
\end{remark}
\begin{example}\rm
Even for $F:C\sr C'$ which commutes with finite coproducts the functor $F^{rad}$ need not respect all projective equivalences between all objects of $\Delta^{op}Rad(C)$. Let $F:C^{\amalg_{<\infty}}\sr C$ be the obvious functor. Using the equivalence between $Rad(C^{\amalg_{<\infty}})$ and presheaves on $C$ one can see that $F^{rad}$ is the radditivization functor of  Proposition \ref{rad2}. The fact that it need not respect projective equivalences is demonstrated in Example \ref{2007grnot}.
\end{example}

\begin{theorem}
\llabel{2009hce}
For any set of morphisms $E$ in $\Delta^{op}Rad(C)$ the localization 
$$H(C,E)=H(C)[cl_l(E)^{-1}]$$
exists and the projection functor $H(C)\sr H(C,E)$ has a right adjoint which identifies $H(C,E)$ with the full subcategory of $E$-local objects in $H(C)$.
\end{theorem}
\begin{proof}
It follows easily by general arguments from Theorem \ref{exlemma}(\ref{exlemma1},\ref{exlemma3}).
\end{proof}
\begin{cor}
\llabel{strloc}
For any $E$ the functor
$$\Delta^{op}C^{\#}\sr H(C,E)$$
is a strict localization i.e. any morphism in the target category is isomorphic to the image of a morphism in the source category.
\end{cor}
\begin{proposition}
\llabel{2009sat}
The class $cl_l(E)$ is saturated i.e. it coincides with the class of morphisms which become isomorphisms in $H(C,E)$. 
\end{proposition}
\begin{proof}
A morphism in $\Delta^{op}Rad(C)$ belongs to $cl_l(E)$ if and only if it is mapped to isomorphisms by the family of functors $Hom_{H(C)}(-,Y)$ for $E$-local objects $Y$. By the universal property of localization these functors factor through the projection to $H(C,E)$. Therefore, any morphism which maps to an isomorphism in $H(C,E)$ is mapped to isomorphisms by these functors and therefore belongs to $cl_l(E)$. 
\end{proof}

\begin{remark}
\rm The obvious analog of Corollary \ref{strloc} holds for $\bar{C}$ instead of $C^{\#}$.
\end{remark}

\begin{theorem}
\llabel{2007el1}
Let $F:C\sr Rad(C')$ be a functor.  Let further $E$ be a set of morphisms in $\Delta^{op}C^{\#}$ and $E'$ and a set of morphisms in $\Delta^{op}Rad(C')$ such that
$$F^{rad}(E\amalg Id_C)\subset cl_l(E').$$
Then 
\begin{eq}
\llabel{2009eq3}
F^{rad}(cl_l(E)\cap \Delta^{op}C^{\#})\subset cl_l(E').
\end{eq}
In particular $\LL F^{rad}(cl_l(E))\subset cl_l(E')$ and $\LL F^{rad}$ defines a functor 
$$H(C,E)\sr H(C',E').$$
\end{theorem}
\begin{proof}
It is sufficient to prove (\ref{2009eq3}). By Theorem \ref{2009main} we have
$$cl_l(E)\cap \Delta^{op}C^{\#}=cl_{\tdl}(E\amalg Id_C)$$
By Lemma \ref{fandt}, our assumption and Theorem \ref{2009innov2} we get
$$F^{rad}(cl_{\tdl}(E\amalg Id_C))\subset cl_{\tdl}(F^{rad}(E\amalg Id_C))\subset cl_{\tdl}(cl_l(E'))=cl_l(E')$$
\end{proof}
\begin{theorem}
\llabel{2007eadj}
Let $F:C\sr C'$ be a functor which commutes with finite coproducts. Let $E$ be a set of morphisms in $\Delta^{op}C^{\#}$ and $E'$ a set of morphisms in $\Delta^{op}(C')^{\#}$. Assume further that one has:
$$F^{rad}(E\amalg Id_C)\subset cl_l(E')$$
$$F_{rad}(E'\amalg Id_{C'})\subset cl_l(E)$$
Then 
\begin{eq}
\llabel{2007in1}
F^{rad}(cl_l(E)\cap \Delta^{op}C^{\#})\subset cl_l(E')
\end{eq}
\begin{eq}
\llabel{2007in2}
F_{rad}(cl_l(E'))\subset cl_l(E)
\end{eq}
In particular, ${\bf L}F^{rad}(cl_l(E))\subset cl_l(E')$ and ${\RR}F_{rad}(cl_l(E'))\subset cl_l(E)$ and the resulting  functors between $H(C,E)$ and $H(C',E')$ are adjoint.
\end{theorem}
\begin{proof}
It is enough to prove the inclusions (\ref{2007in1}) and (\ref{2007in2}). The first one follows from Theorem \ref{2007el1}. For the second inclusion we have 
$$F_{rad}(cl_l(E'))=F_{rad}(cl_{\tdl}((E'\amalg Id_{C'})\cup W_{proj}))\subset cl_{\tdl}(F_{rad}((E'\amalg Id_{C'})\cup W_{proj}))\subset $$
$$\subset cl_{\tdl}(cl_l(E))\subset cl_l(E)$$
where the first equality follows from Theorem \ref{2009main2}, the first inclusion from Lemma \ref{fandt}, the second inclusion from our assumption and the fact that $F_{rad}(W_{proj})\subset W_{proj}$ and the last inclusion from Theorem \ref{2009innov2}. 
\end{proof}
\begin{cor}
\llabel{2007anfe}
Under the assumptions of the theorem one has:
\begin{enumerate}
\item the functor $F_{rad}$ takes $E'$-local objects to $E$-local objects.
\item if $F$ is a full embedding then $\LL F^{rad}:H(C,E)\sr H(C',E')$ is a full embedding and ${\RR}F_{rad}:H(C',E')\sr H(C,E)$ is a localization,
\item if $F$ is surjective on the isomorphism classes of objects then ${\RR}F_{rad}$ reflects isomorphisms i.e.
$$cl_l(E')=F_{rad}^{-1}(cl_l(E)).$$

\end{enumerate}
\end{cor}
\begin{proof}
For $f:X\sr Y$ and $Z'$ we have
$$Hom_{H(C)}(Y,F_{rad}(Z'))=Hom_{H(C')}(\LL F^{rad}(Y), Z')$$
$$Hom_{H(C)}(X,F_{rad}(Z'))=Hom_{H(C')}(\LL F^{rad}(X), Z')$$
and the map between the left hand sides defined by $f$ coincides with the map on the right hand sides defined by $\LL F^{rad}(f)$. If $f$ is in $E$ then $\LL F^{rad}(f)$ is in $E'$ and if $Z'$ is $E'$-local this map is a bijection, which means that $F_{rad}(Z')$ is $E$-local. This proves the first assertion. 

To prove the second assertion  it is sufficient to verify that the adjunction $Id\sr {\RR}F_{rad}\LL F^{rad}$ is an isomorphism in $H(C,E)$. Since any object of $H(C,E)$ is isomorphic to the image of an object from $\Delta^{op}C^{\#}$ and on such objects $\LL F^{rad}=F^{rad}$ it follows from the fact that $Id\sr F_{rad}F^{rad}$ is an isomorphism.

To prove the third assertion note first that since $F$ is surjective on isomorphism classes of objects one has $W_{proj}=F_{rad}^{-1}(W_{proj})$. 
Theorem \ref{2007eadj} implies the inclusion $"\subset"$. Let $f':X'\sr Y'$ be such that $F_{rad}(f')\in cl_l(E)$. We need to show that $f'\in cl_l(E')$. Consider the commutative diagram
$$
\begin{CD}
X' @<<< Cof(X') @>>> Ex_{E'}(Cof(X'))\\
@Vf'VV @Vg'VV @VVh'V\\
Y' @<<< Cof(Y') @>>> Ex_{E'}(Cof(Y'))
\end{CD}
$$
where $Cof$ is the standard cofibrant replacement functor,  the arrows going to the right are in $cl_l(E')$ and the arrows going to the left are projective equivalences. Since $F_{rad}(f)\in cl_l(E)$, Theorem \ref{2007eadj} implies that $F_{rad}(h')\in cl_l(E)$. Then by the first assertion of the corollary $F_{rad}(h')$ is an $E$-local equivalence between $E$-local objects and therefore a projective equivalence. Since $F_{rad}$ reflects projective equivalences we conclude that $h'$ is a projective equivalence and therefore $f'\in cl_l(E')$.
\end{proof}
\begin{remark}
\rm For any adjoint pair of functors such that one functor of the pair is surjective on isomorphism classes of objects the other one reflects isomorphisms. The issue in the proof of Corollary \ref{2007anfe}(3) is that while $F$ is surjective on isomorphism classes of objects the simplicial extension of $F^{rad}$ or even the simplicial extension of $F$ itself need not have this property since there may be many more morphisms and therefore many more simplicial objects in $C'$ then in $C$.
\end{remark}
\begin{remark}\rm
Even when left  Bousfield localizations of the projective c.m.s.'s  with respect to $E$ and $E'$ exist it, is not clear in general whether or not $(F^{rad}, F_{rad})$ is a Quillen adjunction between the localized model categories. 
\end{remark}

\end{document}